\documentclass[11pt,a4paper]{amsart}

\usepackage[utf8]{inputenc}
\usepackage[english]{babel}
\usepackage{hyperref}

\usepackage[foot]{amsaddr}

\usepackage{amsmath,amsfonts,amssymb}
\usepackage{lmodern}
\usepackage[left=2.8cm,right=2.8cm,top=2.8cm,bottom=2.8cm]{geometry}

\usepackage{enumerate}

\usepackage{mathtools}
\usepackage{bm}         % per i caratteri calligrafici in bold
\DeclareMathAlphabet{\mathbbold}{U}{bbold}{m}{n}    % nuovo alfabeto per i numeri in bold.

\usepackage{cancel}

\usepackage[usenames,dvipsnames]{xcolor}
\newtheorem{theorem}{Theorem}
\newtheorem{lemma}[theorem]{Lemma}
\newtheorem{proposition}[theorem]{Proposition}
\newtheorem{corollary}[theorem]{Corollary}
\theoremstyle{definition}

\theoremstyle{remark}
\newtheorem{rmk}{Remark}
\newtheorem{remark}[rmk]{Remark}
\newtheorem{example}{Example}
\newtheorem*{example*}{Example}

\newcommand{\bqn}{\begin{equation}}
\newcommand{\eqn}{\end{equation}}
\newcommand{\distr}{\mathcal{D}}
\newcommand{\g}{\gamma}

\newcommand{\R}{\mathbb{R}}

\newcommand{\eps}{\varepsilon}
\newcommand{\JJ}{\mathcal{J}}
\newcommand{\lam}{\lambda}
\newcommand{\ver}{\mathcal{V}}
\newcommand{\mc}[1]{\mathcal{#1}}
\newcommand{\wt}{\widetilde}
\newcommand{\nn}{\nonumber}

\newcommand{\id}{\mathrm{id}}

\DeclareMathOperator{\spn}{\mathrm{span}}
\DeclareMathOperator{\trace}{\mathrm{tr}}

\allowdisplaybreaks

\title[Bonnet-Myers for quaternionic contact structures]{A Bonnet-Myers type theorem for quaternionic contact structures}

\date{\today}

\author{Davide Barilari$^{\flat}$}

\address{$^{\flat}$ Institut de Math\'ematiques de Jussieu-Paris Rive Gauche, UMR CNRS 7586, Universit\'e Paris-Diderot,
Batiment Sophie Germain, Case 7012, 75205 Paris Cedex 13, France}
\email{davide.barilari@imj-prg.fr}

\author{Stefan Ivanov$^{\sharp}$}

\address{$^{\sharp}$ University of Sofia, Faculty of Mathematics and Informatics, blvd. James Bourchier 5,
1164, Sofia, Bulgaria, \& Institute of Mathematics and Informatics, Bulgarian Academy of Sciences}
\email{ivanovsp@fmi.uni-sofia.bg }

\begin{document}

\begin{abstract}
We prove a Bonnet-Myers type theorem for quaternionic contact manifolds of dimension bigger than 7. If the manifold is complete with respect to the natural sub-Riemannian distance and satisfies a natural Ricci-type bound expressed in terms of derivatives up to the third order of the fundamental tensors, then the manifold is compact and we give a sharp bound on its sub-Riemannian diameter.
\end{abstract}

\maketitle
\tableofcontents

\section{Introduction and main results}\label{s:intro}

Bonnet-Myers theorem is classical among comparison theorems in Riemannian geometry  \cite{myers}. It states that, if the Ricci curvature of a complete $d$-dimensional  Riemannian manifold $(M,g)$ is bounded below by $(d-1)\kappa > 0$, then the manifold $M$ is compact and its diameter is at most $\pi/\sqrt{\kappa}$.

Several generalizations of this theorem, in variuos smooth settings (and even in the non-smooth one of metric measure spaces, see for instance \cite{ohta}) have been recently investigated, introducing suitable notion of curvature or Ricci bound. Among these, different versions of Bonnet-Myers theorem have been obtained in the setting of sub-Riemannian geometry (cf.\ discussion in Section \ref{s:discuss}).

Recall that a sub-Rieman\-nian structure $(\distr,g)$ on a smooth, connected manifold $M$ of dimension $d \geq 3$ is defined by   a vector distribution $\distr$ of constant rank $k \leq d$ and a smooth metric $g$ assigned on $\distr$. The distribution is required to satisfy the H\"ormander condition, or to be bracket-generating, that means
\begin{equation}\label{eq:bg}
%\mathrm{Lie}(\distr)_q = T_q M, \qquad \forall q \in M,
\spn\{[X_{j_1},[X_{j_2},[\ldots,[X_{j_{m-1}},X_{j_m}]]]](x)\mid m \geq 1\} = T_x M, \qquad \forall x \in M,
\end{equation}
for some (and then any) set $X_1,\ldots,X_k \in \Gamma(\distr)$ of local generators for $\distr$.

Given a sub-Riemannian structure on $M$, the \emph{sub-Rieman\-nian distance} is defined by:
\begin{equation*}
d_{SR}(x,y) = \inf\{\ell(\gamma)\mid \gamma(0) = x,\, \gamma(T) = y,\, \gamma \text{ horizontal} \}.
\end{equation*}
where a Lipschitz continuous path $\gamma : [0,T] \to \R$ is \emph{horizontal} if it satisfies $\dot\gamma(t) \in \distr_{\gamma(t)}$ for almost every $t$, and in this case we set
\begin{equation*}
\ell(\gamma) = \int_0^T \sqrt{g(\dot\gamma(t),\dot\gamma(t))}dt.
\end{equation*}
By the classical Chow-Rashevskii theorem (see for instance \cite[Chapter 3]{nostrolibro}), the condition \eqref{eq:bg} implies that $d_{SR}$ is finite and continuous on $M\times M$. We say that the sub-Rieman\-nian manifold is complete if $(M,d_{SR})$ is complete as a metric space.

A sub-Riemannian Bonnet-Myers theorem states, under suitable curvature conditions, that the manifold $M$ is compact and gives a bound on its sub-Riemannian diameter. For more details on sub-Riemannian geometry we refer to classical references such as \cite{bellaiche,montgomerybook} and the more recent ones \cite{nostrolibro,notejean,noterifford}.

\begin{remark} \label{r:one} Notice that if the sub-Riemannian structure is defined as the restriction of a Riemannian metric $g$ on $M$ to a distribution $\distr$, in general the sub-Riemannian diameter is bigger than the Riemannian one. Thus, even if one is able to control the Riemannian curvature of $(M,g)$ and apply a classical Bonnet-Myers theorem, one can prove compactness of $M$, but has no a priori estimate on the sub-Riemannian diameter.
\end{remark}

In this paper we focus on quaternionic contact structure.  A quaternionic
contact (qc) structure, introduced in \cite{Biq1}, appears
naturally as the conformal boundary at infinity of the
quaternionic hyperbolic space. The qc structure gives a natural geometric setting
for the quaternionic contact Yamabe
problem, \cite{GV,Wei,IMV,IMV1}. A particular case of this problem
amounts to find the extremals and the best constant in the $L^2$
Folland-Stein Sobolev-type embedding, \cite{F2} and \cite{follandstein}, with a
complete description of the extremals and the best constant  on
the  quaternionic Heisenberg groups \cite{IMV1,IMV3,IMV4}. 

A quaternionic contact structure  carry a natural sub-Riemannian structure with a codimension three distribution. Curvature conditions are expressed in terms on bounds on standard curvature tensors of quaternionic contact geometry. These conditions can be expressed only in terms of sub-Riemannian quantities (cf.\ Theorems \ref{t:main1} and \ref{t:main2}) and are obtained through the computation of the sub-Riemannian coefficients of the generalized Jacobi equation, first introduced in \cite{geometryjacobi1,lizel} and subsequently developed in \cite{BR-comparison,BR-connection,curvature}.

\subsection{Quaternionic contact structure}  
A quaternionic contact  manifold $(M, \mathbb{Q},g)$ is a $(4n+3)$%
-dimensional manifold $M$ with a codimension-three distribution
$\distr$  equiped with $\mathrm{Sp}(n)\mathrm{Sp}(1)$ structure. Explicitly, the distribution $\distr$  is locally described as
the kernel of a  1-form $\eta=(\eta_1,\eta_2,\eta_3)$ with
values in $\mathbb{R}^3$ together with a compatible Riemannian
metric $g$ and a rank-three bundle $\mathbb{Q}$ consisting of
endomorphisms of $\distr$ locally generated by three almost complex
structures $I_1,I_2,I_3:\distr\to \distr$ satisfying the identities of the
imaginary unit quaternions. Namely, if $\{\alpha,\beta,\tau\}$ is any cyclic permutation of $\{1,2,3\}$ we have
\bqn \label{eq:relaz1}
I_\alpha I_\beta=-I_\beta I_\alpha=I_\tau,
\quad
I_{\alpha}^{2}=I_{\beta}^{2}=I_{\tau}^{2}=I_\alpha I_\beta I_\tau=-\id_{|_\distr}.
\eqn 
Moreover $I_1,I_2,I_3$ are compatible with the metric $g$, in the following sense: for every $\alpha=1,2,3$ and $X,Y\in \distr$ we have
\begin{equation*} \label{eq:relaz2}
g(I_{\alpha}X,I_{\alpha}Y)=g(X,Y),\qquad
2g(I_{\alpha}X,Y) = d\eta_{\alpha}(X,Y).
\end{equation*}

From the sub-Riemannian view-point, these structures are \emph{fat}, i.e. for any non zero section $X$ of $\distr$, $TM$ is (locally) generated by $\distr$ and $[X,\distr]$. This is a direct consequence of the quaternionic relations of the almost complex structures. For completeness a proof is given in Section \ref{s:2}. The fat condition is open in the $C^1$ topology, however it gives some restriction on the rank $k$ of the distribution (for example $\dim M \leq 2k -1$, \cite[Prop.\ 5.6.3]{montgomerybook}).

\begin{example}[Quaternionic Hopf fibration] \label{ex:qhf}
A classical example of quaternionic contact structure is the quaternionic Hopf fibration
\begin{equation}
\mathbb{S}^3 \hookrightarrow \mathbb{S}^{4n+3} \xrightarrow{\pi} \mathbb{HP}^n, \qquad n \geq 1.
\end{equation}
Here $\distr = (\ker \pi_*)^\perp$ is the orthogonal complement of the kernel of the differential of the Hopf map $\pi$, and the sub-Rieman\-nian metric is the restriction to $\distr$ of the standard round metric on $\mathbb{S}^{4n+3}$.  The sub-Riemannian distance on the quaternionic Hopf fibration can be computed explicitly and its diameter is $\pi$, as it is proved in \cite{BW14}. This example is one of the simplest (non-Carnot) sub-Rieman\-nian structures of corank greater than $1$, and is included in the sub-class of \emph{$3$-Sasakian} structures.
\end{example}

\begin{example}[Quaternionic Heisenberg group]
An example of quaternionic contact structure that is not 3-sasakian is the quaternionic Heisenberg group. It is defined as
$$\R^{4n+3}=\mathbb{H}^{n}\oplus \mathrm{Im}(\mathbb{H})$$
endowed with the group law
$$(z,w)\cdot(z',w')=\left(z+z',w+w'+\frac12 \mathrm{Im}(z\bar z')\right).$$
If we take $\distr=\mathbb{H}^{n}$ (which has dimension $4n$) with the standard Euclidean metric, it easy to see that it is bracket generating and defines a quaternionic contact structure.
\end{example}

\subsection{Biquard connection, torsion and curvature}

On a qc manifold of dimension $4n+3$ with $n\geq 2$ with a fixed metric $g$
on the \emph{horizontal distribution} $\distr$ there exists a canonical
connection, called \emph{Biquard connection}, defined in
\cite{Biq1}.  Biquard shows that there exists a unique supplementary subspace $V$ to $\distr$ in $%
TM$ and a unique connection $%
\nabla$ with torsion $T$, such that:

\begin{itemize}

\item[(i)] $\nabla$ preserves the decomposition $H\oplus V$ and the
$\mathrm{Sp}(n)\mathrm{Sp}(1)$ structure on $\distr$,  $\nabla g=0$, $\nabla\sigma
\in\Gamma( \mathbb{Q})$ for $\sigma\in\Gamma(\mathbb{Q})$, and
its torsion on $\distr$ is given by $T(X,Y)=-[X,Y]_{|V}$;

\item[(ii)] for $\xi\in V$, the endomorphism $T(\xi,.)_{|\distr}$ of $\distr$ lies in\footnote{the perpendicular is computed with respect to the inner product
%\begin{equation*}
$\langle A \,|\, B\rangle = {\ \sum_{i=1}^{4n} g(A(e_i),B(e_i)),}$
%\end{equation*}
for $A, B \in \text{End}(H)$.} $%
(\mathrm{sp}(n)\oplus \mathrm{sp}(1))^{\bot}\subset \mathrm{gl}(4n)$;

\item[(iii)] the connection on $V$ is induced by the natural
identification $\varphi
$ of $V$ with the subspace $\mathrm{sp}(1)$ of the endomorphisms of $\distr$, i.e., $%
\nabla\varphi=0$.
\end{itemize}

%\end{thrm}

When the dimension of $M$ is at least eleven, \cite{Biq1} shows that 
the supplementary \emph{vertical distribution} $V$ is
(locally) generated by three \emph{Reeb vector fields}
$\xi_1,\xi_2,\xi_3$ determined by the conditions
\begin{equation}  \label{reeb1}
\begin{aligned}
\eta_{\alpha}(\xi_{\beta})=\delta_{{\alpha}{\beta}}, \qquad
(\xi_{\alpha}\lrcorner d\eta_{\alpha})_{|_{\distr}}=0, \qquad
(\xi_{\alpha}\lrcorner d\eta_{\beta})_{|_{\distr}}=-(\xi_{\beta}\lrcorner
d\eta_{\alpha})_{|_{\distr}},
\end{aligned}
\end{equation}
where $\lrcorner$ denotes the interior multiplication: more explicitly $X\lrcorner \Phi=\Phi(X,\cdot)$ where $X$ is a vector field and $\Phi$ is a differential 2-form.

\begin{remark} In this paper we restrict our attention to quaternionic contact structure of dimension strictly bigger that seven.
If the dimension of $M $ is seven Duchemin shows in \cite{D} that
if we assume, in addition, the existence of Reeb vector fields as
in \eqref{reeb1}, then the Biquard result holds. Henceforth,
by a qc structure in dimension $7$ we shall mean a qc structure
satisfying \eqref{reeb1}. This implies the existence of the
connection with properties (i), (ii) and (iii) above.
\end{remark}
The fundamental 2-forms $\omega_{\alpha}$ of the qc structure are
defined by 
\begin{equation*}
2\omega_{{\alpha}|\distr} =d\eta_{{\alpha}|\distr},\quad
\xi\lrcorner\omega_{\alpha}=0,\quad \xi\in V. 
\end{equation*}
The torsion restricted to $\distr$ has the form 
$$
T(X,Y)=-[X,Y]_{|V}=2\sum_{{\alpha}=1}^3\omega_{\alpha}(X,Y)\xi_{\alpha}
.$$

The properties of the Biquard connection are encoded in the torsion endomorphism $T(\xi,.)_{|\distr}$. It is  completely trace-free, $\mathrm{tr}( T(\xi,.)_{|\distr})=\mathrm{tr}T(\xi,.)_{|\distr}\circ I_{\alpha}=0$  and  can be decomposed into symmetric and skew-symmetric parts,  $T(\xi_{\alpha},.)_{|\distr}= T^0(\xi_{\alpha},.)_{|\distr}+I_{\alpha}u$, respectively where $u$
is a traceless symmetric (1,1)-tensor on $\distr$ which commutes with $I_{1},I_{2},I_{3}$, see \cite{Biq1}. When $n=1$ the tensor $u$ vanishes identically and the torsion is a symmetric tensor, $T_{\xi }=T_{\xi }^{0}$.

 The two $\mathrm{Sp}(n)\mathrm{Sp}(1)$-invariant trace-free symmetric 2-tensors $T^0, U$ on
$\distr$ defined in \eqref{Tcompnts} having the properties \eqref{propt}
determine completely  the symmetric and the skew-symmetric parts of torsion endomorphism, respectively \cite{IMV} (cf.\ \eqref{need} and \eqref{need1} in the Appendix.)

The \emph{qc-Ricci tensor} $\mathrm{Ric}$ and the
\emph{normalized qc-scalar curvature} $S$  of the Biquard connection are defined with the usual horizontal traces of the curvature of the Biquard connection (cf.\  \eqref{qscs} in the Appendix.)

 A qc structure is said to be \emph{qc-Einstein} if the horizontal
qc-Ricci tensor is a scalar multiple of the metric.
As shown in \cite{IMV,IMV2} the qc-Einstein condition is
equivalent to the vanishing of the torsion endomorphism of the
Biquard connection. In this case $S$ is constant and the vertical
distribution is integrable. It is also worth recalling that the
horizontal qc-Ricci tensors and the integrability of the vertical
distribution can be expressed in terms of the torsion of the
Biquard connection according to \eqref{sixtyfour} in the Appendix  (see \cite{IMV}, cf.\ also \cite{IMV1,IV,IPV1}) .

Any 3-Sasakian manifold has zero torsion endomorphism, and the
converse is locally true if in addition the qc scalar curvature  is a positive constant \cite{IMV}.

\subsection{Main results}
Our main result reads as follows, in terms of  the qc Ricci tensor and the curvature tensor associated with the Biquard connection.
\begin{theorem}\label{Bone-mayers}\label{t:main1}
Let  $(M, g, \mathbb{Q})$ be a $4n+3$-dimensional complete qc manifold with $n>1$. Assume that there exists a constant  $\kappa>0$ such that
\bqn \label{eq:trace00}
\mathrm{Ric}(X,X)- \sum_{\alpha=1}^3 R(X,I_\alpha X,I_\alpha X,X)\ge 4(n-1)\kappa,\qquad \forall\, X\in \distr.
\eqn
Then   $(M, g, \mathbb{Q})$ is compact manifold with finite fundamental group, and its sub-Riemannian diameter is not greater than $\pi/\sqrt{\kappa}$.
\end{theorem}
\begin{remark} 
The bound on the sub-Riemannian diameter given in Theorem \ref{t:main1} is sharp since the equality is attained for the quaternionic Hopf fibration, where $\kappa$ can be chosen equal to 1 and the sub-Riemannian diameter is $\pi$ (cf.\ Example \ref{ex:qhf} and Remark \ref{r:qhf}).
Moreover, since $R(X, X, X,X)=0$, the left hand side in \eqref{eq:trace00} is indeed a trace on a $4(n-1)$-dimensional subspace of $\distr$. 
\end{remark}
Theorem \ref{t:main1} can be also restated as follows, in terms of the horizontal part of torsion tensors and scalar curvature.
\begin{theorem}\label{Bone-mayers1}\label{t:main2}
Let  $(M, g, \mathbb{Q})$ be a $4n+3$-dimensional complete qc manifold with $n>1$. Assume that there exists a constant  $\kappa>0$ such that
\bqn \label{eq:main2}
2nT^0(X,X)+(4n-8)U(X,X)+2(n-1)S\ge 4(n-1)\kappa,\qquad \forall\, X\in \distr.
\eqn
Then  $(M, g, \mathbb{Q})$ is compact manifold with finite fundamental group, and its sub-Riemannian diameter not
greater than $\pi/\sqrt{\kappa}$.
\end{theorem}
The proofs of Theorems \ref{t:main1} and \ref{t:main2} are given in Section \ref{s:proofmain}.
We stress that, thanks to the results in \cite{IMV} and the proof of \cite[Theorem~4.2.5]{IV2}, the condition \eqref{eq:main2} can be rewritten only in terms of the qc structure and its Lie derivatives.

\begin{proposition} Let $\{X_{1},\ldots,X_{4n}\}$ be a local orthonormal basis for $\distr$ and $\{\alpha,\beta,\tau\}$ be any cyclic permutation of $\{1,2,3\}$. We have the following relations:
\begin{itemize}
\item[(i)] The symmetric part of the torsion endomorphism is determined entirely
by the Lie derivative of the metric
$$ T^{0}(\xi,X,Y)=\frac12 \mathcal{L}_{\xi}g(X,Y), \qquad  T^0(X,Y)=\frac12\sum_{\alpha=1}^3(\mathcal{L}_{\xi_{\alpha}}g)(I_{\alpha}X,Y).$$
\item[(ii)] The skew-symmetric part of the torsion described by $U$ satisfies
\begin{align*}
U(X,Y)=&\frac14g((\mathcal{L}_{\xi_{\beta}}I_{\alpha})X,I_{\tau}Y)-\frac14g((\mathcal{L}_{\xi_{\beta}}I_{\alpha})I_{\tau}X,Y)\\
&\qquad +\frac{1}{2n}\sum_{i=1}^{4n}g((\mathcal{L}_{\xi_{\beta}}I_{\alpha})I_{\tau}X_{i},X_{i})g(X,Y),
\end{align*}
\item[(iii)] The  normalized qc  scalar curvature is written as
$$
S=d\eta_{\alpha}(\xi_{\beta},\xi_{\tau})-d\eta_{\tau}(\xi_{\alpha},\xi_{\beta})-d\eta_{\beta}(\xi_{\tau},\xi_{\alpha})-\frac{1}{2n}\sum_{i=1}^{4n}g((\mathcal{L}_{\xi_{\beta}}I_{\alpha})I_{\tau}X_{i},X_{i}).$$
\end{itemize}
\end{proposition}

\begin{remark} We note that Theorems \ref{t:main1} and \ref{t:main2}  generalize  Bonnet-Myers results for the sub-class of qc manifold with integrable vertical space obtained in \cite{hladky} simplifying  considerably the Bonnet-Myers positivity condition and giving moreover explicit diameter bounds.
\end{remark}

\begin{remark}[3-Sasakian case]\label{r:qhf}
Assume that the qc manifold is 3-Sasakian.
In this case, we have from \cite[Corollary~4.13]{IMV} and \cite[Theorem~3.12]{IMV} that
$$T^0=U=0,\qquad \quad S=2.$$
Therefore, \eqref{eq:main2} is satisfied with $\kappa=1$ and we recover the universal diameter bound for 3-Sasakian manifold, established in \cite{RS-3-Sasakian}.  The diameter bound is attained for the quaternionic Hopf fibration (cf.\ Example \ref{ex:qhf}).

Notice that in \cite{RS-3-Sasakian} the authors use curvature tensors $R^{g}$ associated with the Levi-Civita connection. Using the relation between $R^{g}$ and the curvature tensor $R$  associated with the Biquard connection (see \cite[Corollary~4.13]{IMV}) we have
\begin{equation*}\label{birim}\sum_{\alpha=1}^3 R(X,I_\alpha X,I_\alpha X,X)=\sum_{\alpha=1}^3 R^{g}(X,I_\alpha X,I_\alpha X,X)+9=12,
\end{equation*}
where we apply the identity $\sum_{\alpha=1}^3 R^{g}(X,I_\alpha X,I_\alpha X,X)=3$, valid for 3-Sasakian manifolds, and proved in \cite[Prop.~3.2]{tanno71}.

\end{remark}

We also state the following interesting corollary of Theorem \ref{t:main2}, when  $n=2$.
\begin{corollary}
Let  $(M, g, \mathbb{Q})$ be a $11$-dimensional complete qc manifold. Assume $T^0=0$ and $S\ge 2\kappa>0$, then $M$ is compact with sub-Riemannian diameter not greater than $\pi/\sqrt \kappa$.
\end{corollary}

We note that qc manifolds with $T^0=0$ are characterized with the condition that the  almost contact structure on the corresponding  twistor space is normal, see \cite{DIM}.

\subsection{Relation with previous literature} \label{s:discuss} Other sub-Riemannian Bonnet-Myers type results are found in the literature, proved with different techniques and for different sub-Riemannian structures. The three dimensional contact case has been considered using second variation like formulas in \cite{Rumin} (for CR structures) and in \cite{Hughen-PhD}. In \cite{garofalob} and \cite{BKW-weitzenbock}, a version of Bonnet-Myers has been proved using heat semigroup approach for Yang-Mills type structures with transverse symmetries and Riemannian foliations with totally geodesic leaves, respectively. Using Riccati comparison techniques for 3D contact \cite{AAPL}, 3-Sasakian \cite{RS-3-Sasakian} and general contact sub-Riemannian structure \cite{ABR-contact}. See also \cite{BR-comparison} for a general approach to sub-Riemannian Bonnet-Myers theorem through curvature invariants.

A compactness result, obtained by applying Riemannian classical Bonnet-Myers theorem to a suitable Riemannian extension of the metric (cf.\ Remark \ref{r:one}) is obtained in \cite{baudoincontact} for contact manifolds and in \cite{hladky} for quaternionic contact manifolds with integrable vertical space.

\subsection{Structure of the paper} In Section \ref{s:2} we recall some results about the sub-Riemannian Jacobi equation and the curvature invariants.  In Sections \ref{s:3} and \ref{s:5} we carefully compute these invariants for quaternionic contact structures and express them with respect to standard tensors of quaternionic contact geometry. To perform these computations, we introduce a generalized Fermi frame along the geodesic. In Section \ref{s:proofmain} we use these computations to prove the Bonnet-Myers theorem. Appendix \ref{s:quatapp} resumes geometric properties of quaternionic contact structures.

\section{Curvature of sub-Riemannian qc structures} \label{s:2}
In this section we resume the basic facts on sub-Riemannian geodesic flows and curvature needed to prove our results. For a more comprehensive presentation we refer the reader to \cite{curvature,BR-comparison,BR-connection}.
\subsection{Quaternionic contact sub-Riemannian structures are fat}
A sub-Riemannian structure is said to be \emph{fat} if for any non zero section $X$ of $\distr$, $TM$ is (locally) generated by $\distr$ and $[X,\distr]$. This is equivalent to show that for every non zero horizontal vector $X\in \distr$ the following map is surjective
$$\mc{L}_{X}:\distr\to TM/\distr,\qquad \mc{L}_{X}(Y):=[X,Y] \mod \distr$$
Notice that the map $\mc{L}_{X}$ is tensorial, in the sense that for each $x\in M$ the value of $[X,Y](x) \mod \distr_{x}$ depends only on $X(x)$ and $Y(x)$. Moreover $\dim T_{x}M/\distr_{x}=3$.
The fat property follows from the following linear algebra observation.
\begin{lemma} The vectors $\{\mc{L}_{X}(I_{\alpha}X)\}_{\alpha=1,2,3}$ are linearly independent in $T_{x}M/\distr_{x}$.
\end{lemma}
\begin{proof} Let us start by showing that, for every non zero horizontal vector $X\in \distr$ the four vectors $X,I_{1}X,I_{2}X,I_{3}X$ are mutually orthogonal. First notice that for every $\tau=1,2,3$ and every horizontal vector $X\in \distr$ one has
\bqn\label{eq:ppp}
g(I_{\tau}X,X)=d\eta(X,X)=0.
\eqn
Moreover, if $\{\alpha \beta \tau\}$ is a cyclic permutation of $\{1 2 3\}$, thanks to \eqref{eq:relaz1} and \eqref{eq:ppp}, one has for $\alpha\neq \beta$
$$g(I_{\alpha}X,I_{\beta}X)=-g(X,I_{\alpha}I_{\beta}X)=-g(X,I_{\tau}X)=0.$$

To prove that the sub-Riemannian structure is fat it is sufficient to show that the image through $L_{X}$ of the vectors $I_{\beta}X$ for $\beta=1,2,3$ is a basis of $T_{x}M/\distr_{x}$, for every $x\in M$. This is equivalent to say that for every $X$ the matrix $\Omega_{\alpha\beta}=\eta_{\alpha}(L_{X}(I_{\beta}X))$ is invertible, which follows from 
$$\Omega_{\alpha\beta}=\eta_{\alpha}(\mc{L}_{X}(I_{\beta}X))=-d\eta_{\alpha}(X,I_{\beta}X)=-2g(I_{\alpha}X,I_{\beta}X)=-2\delta_{\alpha\beta}g(X,X).\hfill \qedhere$$
\end{proof}

\subsection{Sub-Riemannian geodesic flow} 
Sub-Rieman\-nian \emph{geodesics} are horizontal curves that are locally minimizers for the length (between curve with same endpoints). The \emph{sub-Rieman\-nian Hamiltonian} $H : T^*M \to \R$ is defined as
\begin{equation*}
H(\lambda) := \frac{1}{2}\sum_{i=1}^k \langle \lambda, X_i \rangle^2, \qquad \lambda \in T^*M,
\end{equation*}
where $X_1,\ldots,X_k$ is any local orthonormal frame for $\distr$ and $\langle \lambda, v \rangle $ denotes the action of a covector $\lambda\in T_{x}^{*}M$ on a vector $v\in T_{x}M$, based at $x\in M$. Let $\sigma$ be the canonical symplectic form on $T^*M$. The \emph{Hamiltonian vector field} $\vec{H}$ is defined by the identity $\sigma(\cdot, \vec{H}) = dH$. Then the Hamilton equations are
\begin{equation}\label{eq:Hamiltoneqs}
\dot{\lambda}(t) = \vec{H}(\lambda(t)).
\end{equation}
Solutions of~\eqref{eq:Hamiltoneqs} are called \emph{extremals}, and one can prove that their projections $\gamma(t) := \pi(\lambda(t))$ on $M$ are geodesics \cite[Chapter 4]{nostrolibro}. The Hamiltonian $H$ is constant along an extremal $\lam(t)$ and we say that the extremal is \emph{length-parametrized} if $H(\lambda(t))=1/2$.  

Since the sub-Riemannian structure defined on the qc manifold is fat, any minimizer can be recovered uniquely in this way. This statement is not true in full generality, since there can exist minimizing trajectories might not satisfy the Hamiltonian equation ~\eqref{eq:Hamiltoneqs}. These trajectories are called \emph{abnormal minimizers} and are related to the main open problems in sub-Riemannian geometry (see for instance  \cite{AAA-open} for a discussion).

\subsection{Jacobi equation revisited}
Given an extremal $\lambda(t)$ of the sub-Riemannian Hamiltonian flow and a vector field $V(t)$ along $\lambda(t)$ we define
\begin{equation*}
\dot{V}(t) := \left.\frac{d}{d\eps}\right|_{\eps=0} e^{-\eps \vec{H}}_* V(t+\eps).
\end{equation*}
the Lie derivative of $V$  in the direction of $\vec{H}$. A vector field $\JJ(t)$ along $\lam(t)$ is a \emph{sub-Riemannian Jacobi field} if
\begin{equation}\label{eq:defJFF}
\dot{\JJ} = 0.
\end{equation}
If $M$ has dimension $d$, the set of solutions of \eqref{eq:defJFF} is a $2d$-dimensional vector space. The projections $\pi_{*}\JJ(t)$ are vector fields on the manifold $M$ corresponding to one-parameter variations of $\g(t)=\pi(\lam(t))$ through geodesics. In the Riemannian case, this coincides with the classical construction of Jacobi fields.

Next, let us write \eqref{eq:defJFF} using the symplectic structure $\sigma$ of $T^{*}M$.  Observe that on $T^*M$ there is a natural notion of \emph{vertical subspace} at $\lam\in T^{*}M$, namely
\begin{equation*}
\ver_{\lambda} := \ker \pi_*|_{\lambda} = T_\lambda(T^*_{\pi(\lambda)} M) \subset T_{\lambda}(T^*M).
\end{equation*}
Then $\ver$ is a smooth (Lagrangian) sub-bundle of $T(T^{*}M)$. If one considers the frame $E_{i}=\partial_{p_i}|_{\lambda(t)}$, and $F_{j}=\partial_{x_j}|_{\lambda(t)}$ induced by coordinates $(x_1,\ldots,x_d)$ on $M$, then the vector field $\JJ(t)$ has components $(p(t),x(t)) \in \R^{2d}$, that means
\begin{equation*}
\JJ(t) = \sum_{i=1}^d p_{i}(t) E_{i}(t) + x_{i}(t) F_{i}(t).
\end{equation*}
and the elements of the frame satisfy the equation

\begin{equation}\label{eq:Jacobiframe}
\frac{d}{dt}\begin{pmatrix}
E \\
F
\end{pmatrix} =
\begin{pmatrix}
\mc{A}(t) & -\mc{B}(t) \\
\mc{R}(t) & -\mc{A}(t)^*
\end{pmatrix} \begin{pmatrix}
E\\
F
\end{pmatrix},
\end{equation}
for some smooth families of $d\times d$ matrices $\mc{A}(t),\mc{B}(t),\mc{R}(t)$, where $\mc{B}(t) = \mc{B}(t)^*$ and $\mc{R}(t)= \mc{R}(t)^*$. The structure of \eqref{eq:Jacobiframe} follows from the fact that the frame is Darboux, namely
\begin{equation*}
\sigma(E_i,E_j) = \sigma(F_i,F_j) = \sigma(E_i,F_j) -\delta_{ij} = 0, \qquad i,j=1,\ldots,d.
\end{equation*}

The idea is then to look for a suitable Darboux frame $\{E_i(t),F_i(t)\}_{i=1}^{d}$ along $\lambda(t)$ such that the equations above are in normal form.

\subsection{Curvature coefficients in quaternionic contact}\label{sec:qcnormalframe}

The normal form of the sub-Riemannian Jacobi equation \eqref{eq:Jacobiframe} has been first studied by Agrachev-Zelenko in \cite{agzel1,agzel2} and subsequently completed by Zelenko-Li in \cite{lizel}. In particular, there exist a normal form of \eqref{eq:Jacobiframe} where the matrices $\mc{A}(t)$ and $\mc{B}(t)$ are constant. Here we give an ad-hoc statement for quaternionic contact sub-Riemannian structures, following the notation and the presentation of \cite{BR-connection}.

\begin{rmk}
It is convenient to split the set of indices $1,\ldots,4n+3$ into three subsets $a,b,c$
with cardinality  $|a| = |b| = 3$ and  $|c| = 4n-3$. The index $a$ parametrizes the three-dimensional complement to the distribution, while $b$ and $c$ together  parametrize the set of indices on the distribution.

This splitting is related to the fact that the Lie derivative $\mc{L}_X : \distr_x \to T_x M/\distr_x$ in the direction of a nontrivial horizontal vector $X \in \distr_{q}$ induces a well defined, surjective linear map with $3$-dimensional image (the ``$a$'' space) and a $4n-3$-dimensional kernel (the ``$c$'' space).  The orthogonal complement of the kernel within $\distr_x$ is a $3$-dimensional space (the ``$b$'' space).

Accordingly to this decomposition, any $(4n+3) \times (4n+3)$ matrix $L$ is written in the block form
\begin{equation*}\label{eq:decomposition}
L = \begin{pmatrix}
L_{aa} & L_{ab} & L_{ac} \\
L_{ba} & L_{bb} & L_{bc} \\
L_{ca} & L_{cb} & L_{cc}
\end{pmatrix},
\end{equation*}
with similar notation for row or column vectors.
\end{rmk}
\begin{theorem}\label{p:can}
Let $\lambda(t)$ be a sub-Riemannian extremal of a qc sub-Riemannian structure. There exists a smooth moving frame along $\lambda(t)$
\begin{equation*}
E(t)  = (E_a(t),E_b(t),E_c(t))^*, \qquad F(t)  = (F_a(t),F_b(t),F_c(t))^*,
\end{equation*}
such that the following holds true for any $t$:
\begin{itemize}
\item[(i)] $\spn\{E_a(t),E_b(t),E_c(t)\} = \mathcal{V}_{\lambda(t)}$.
\item[(ii)] It is a Darboux basis, namely
\begin{equation*}
\sigma(E_\mu,E_\nu) = \sigma(F_\mu,F_\nu)= \sigma(E_\mu,F_\nu) - \delta_{\mu\nu} = 0, \qquad \mu,\nu =a,b,c.
\end{equation*}
\item[(iii)] The frame satisfies the \emph{structural equations}
\begin{align*}
\dot{E}_a & = E_b,      & \dot{E}_b & = -F_b,   & \dot{E}_c & = -F_c,  \\
\dot{F}_a & =  \sum_{\mu=a,b,c}\mc{R}_{a \mu}(t)E_{\mu}, &  \dot{F}_b & =  \sum_{\mu=a,b,c}\mc{R}_{b\mu}(t)E_{\mu} - F_a & \dot{F}_c, & =  \sum_{\mu=a,b,c}\mc{R}_{c\mu}(t)E_{\mu}.
\end{align*}

where the curvature matrix $\mc{R}(t) = \mc{R}(t)^*$ is
\begin{equation*}
\mc{R}(t) = \begin{pmatrix}
\mc{R}_{aa}(t) & \mc{R}_{ab}(t) & \mc{R}_{ac}(t) \\
\mc{R}_{ba}(t) & \mc{R}_{bb}(t) & \mc{R}_{bc}(t) \\
\mc{R}_{ca}(t) & \mc{R}_{cb}(t) & \mc{R}_{cc}(t)
\end{pmatrix},
\end{equation*}
and satisfies the additional condition $\mc{R}_{ab}(t) = - \mc{R}_{ab}(t)^*$.

\end{itemize}

\end{theorem}
\begin{remark}
If we fix another frame $\{\wt{E}(t),\wt{F}(t)\}$ satisfying (i)-(iii) for some matrix $\wt{\mc R}(t)$, then there exists a constant $n \times n$ orthogonal matrix $O$ that preserves the structural equations and such that
\begin{equation*}
\wt{E}(t) = O E(t), \qquad \wt{F}(t) = O F(t), \qquad \wt{\mc R}(t) = O \mc R(t) O^*.
\end{equation*}
For more details about the uniqueness of this frame we refer the reader to \cite{BR-connection}.
\end{remark}
\subsection{Ricci curvature and Bonnet-Myers theorem}\label{s:invariantspaces}
We can state now a corollary of the general results obtained in \cite{BR-comparison} (an analogue statement of the one mentioned here is \cite[Thm.\ 5]{RS-3-Sasakian}) that gives a Bonnet-Myers type theorem that we will use to prove our results.
\begin{theorem}\label{t:BRmain} Let $(M,\distr,g)$ be a complete qc sub-Rieman\-nian manifold. Assume that there exists $\kappa>0$ such that for any length-parametrized extremal $\lambda(t)$ one has
\begin{equation*}
\trace (\mc{R}_{cc}(t)) \geq 4(n-1)\kappa,
\end{equation*}
Then $M$ is compact and its sub-Rieman\-nian diameter is bounded by $\pi/\sqrt{\kappa}$.
Moreover $M$ has finite fundamental group.
\end{theorem}

In the following sections we will compute the quantity $\trace(\mc{R}_{cc}(t))$ for every sub-Riemannian extremal on a qc manifold and deduce the main theorems stated in the Introduction.

\section{Structural equations for the coordinate frame} \label{s:4}

\label{s:3}

In what follows latin indices $i,j,k,\dots$ belong to $\{1,\dots,4n\}$ and
Greek ones $\alpha,\beta,\tau,\ldots$ belong to $\{1,2,3\}$, corresponding to quaternions (following the same quaternionic indices notation of Appendix~\ref{s:quatapp}).

 We start by choosing a convenient local frame on $M$, associated with a given
trajectory. Here $\{X_{1},\ldots,X_{4n}\}$ will denote a local orthonormal frame for the metric $g$ on  $\distr$.

\subsection{Fermi frame}
Given a geodesic $\gamma(t) = \pi(\lambda(t))$,  we define a convenient local frame on $M$ which is an application of a standard result in
differential geometry, called Fermi normal frame along a smooth curve.

\begin{lemma}\label{fermi_frame}
Given a geodesic $\gamma(t)$, there exists a $\mathbb Q$-orthonormal frame, i.e., a horizontal
frame $X_i$, $i \in \{1,\dots,4n\}$,
%$$\{X_1,X_2=IX_1,\dots,X_{4n}=KX_{4n-3}\},$$ %
and vertical frame $\xi_{\alpha}$, $\alpha=1,2,3$
in a neighborhood of $\gamma(0)$, such that for all $\alpha, \beta
\in\{1,2,3\}$ and $i,j\in\{1,\dots,4n\}$,
   
   \begin{itemize}
        \item[(i)] the frame is orthonormal for the Riemannian metric $g+\sum_{\beta}\eta_{\beta}^2$,
        \item[(ii)] $\nabla_{X_i} X_j|_{\gamma(t)} = \nabla_{\xi_\alpha} X_j|_{\gamma(t)} =\nabla_{X_i} \xi_\beta|_{\gamma(t)}=\nabla_{\xi_{\alpha}} \xi_\beta|_{\gamma(t)}
        =0$.
    \end{itemize}
     In particular, for all $\alpha, \beta,\tau 
\in\{1,2,3\}$ and $i,j\in\{1,\dots,4n\}$
$$((\nabla_{X_i}I_{\alpha})X_j)|_{\gamma(t)}=((\nabla_{X_i}I_{\alpha})\xi_{\beta})|_{\gamma(t)}=((\nabla_{\xi_{\beta}}I_{\alpha})X_j)|_{\gamma(t)}
=((\nabla_{\xi_{\beta}}I_{\alpha})\xi_{\tau})|_{\gamma(t)}=0.$$
\end{lemma}
The proof of this Lemma is postponed to Appendix \ref{s:fermia}.

\subsection{Commutator relations and Poisson brackets}
Fix  $\{X_1,X_2,\dots$ $,X_{4n}\}$ a horizontal
frame %$X_i$, $i \in \{1,\dots,4d\}$,
and  $\xi_{\alpha}$ for $\alpha=1,2,3$ vertical frame and introduce the \textit{momentum functions} $u_i,v_\alpha: T^*M \to \R$ defined by
\begin{align*}
u_i(\lambda) & =  \langle \lambda, X_i \rangle , \qquad i = 1,\dots,4n, \\
v_\alpha(\lambda) & = \langle \lambda, \xi_\alpha\rangle , \qquad \alpha = 1,2,3.
\end{align*}
The momentum functions define coordinates $(u,v)$ on each fiber of $T^*M$. In turn, they define local vector fields $\partial_{v_\alpha}$ and $\partial_{u_i}$ on $T^*M$ (satisfying $\pi_* \partial_{v_\alpha} = \pi_*\partial_{u_i} = 0$). Moreover, they define also the Hamiltonian vector fields $\vec{u}_i$ and $\vec{v}_\alpha$. The \emph{Hamiltonian frame} associated with $\{\xi_\alpha,X_i\}$ is the local frame on $T^*M$ around $\lambda(0)$ given by $\{\partial_{u_i},\partial_{v_\alpha},\vec{u}_i,\vec{v}_\alpha\}$.

The sub-Riemannian Hamiltonian and the corresponding Hamiltonian vector field are
\begin{equation*}
H = \frac{1}{2} \sum_{i=1}^{4n}u_i u_i, \qquad \vec{H} = \sum_{i=1}^{4n}u_i \vec{u}_i.
\end{equation*}
 We will use the short notation $\alpha\beta=\tau$, where  $ \alpha,\beta=1,2,3$, for the quaternionic multiplication. The following $3\times 3$  skew-symmetric matrix contains the vertical part of the covector:
$$V=V_{\alpha\beta}=v_{\alpha\beta}=v_{\alpha}v_{\beta}$$
with the convention $v_{\alpha^2}=-v_1=0$ which is the standard identification $\R^3\backsimeq \mathfrak{so}(3)$.
%$$ \mathbf{T}(\xi_{\alpha},X,Y):=\frac14[T^{0}(I_{\alpha}X,Y)+T^{0}(X,I_{\alpha}Y)]-U(I_{\alpha}X,Y)$$

\begin{remark}
For functions $f,g \in C^\infty(T^*M)$, the symbol $\{f,g\}$ denotes
their Poisson bracket. The symbol $\dot f$ always denotes the Lie
derivative in the direction of $\vec{H}$. 
We make systematic use of symplectic calculus (see for instance \cite{Agrachevbook} for reference).
\end{remark}

In what follows, we fix a geodesic $\g(t)$ with corresponding lift $\lambda(t)$ and a Fermi frame associated with it and given by Lemma \ref{fermi_frame}. Repeated indices are implicitly summed over.
\begin{lemma}[Commutators] \label{l:commutatori}
We have the following identities
\begin{itemize}
        \item[(a')] $g([X,Y],\xi_{\alpha})=-d\eta_{\alpha}(X,Y)=-2g(I_{\alpha}X,Y)$
        \item[(b')] $g([X,Y],Z)=g(\nabla_{X}Y,Z)-g(\nabla_{Y}X,Z)$
        \item[(c')] $g([\xi_{\alpha},X],\xi_{\beta})=d\eta_{\alpha}(\xi_{\beta},X)=-d\eta_{\beta}(\xi_{\alpha},X)
       =-g(\nabla_X\xi_{\alpha},\xi_{\beta}).$
        \item[(d')] $g([\xi_{\alpha},X],Y)=  -T(\xi_{\alpha},X,Y)+g(\nabla_{\xi_{\alpha}}X,Y)$
        \item[(e')]  $g([\xi_{\alpha},\xi_{\beta}],\xi_{\gamma})=-d\eta_{\gamma}(\xi_{\alpha},\xi_{\beta})=Sg(\xi_{\alpha\beta},\xi_{\gamma}) +g(\nabla_{\xi_{\alpha}}\xi_{\beta},\xi_{\gamma})-g(\nabla_{\xi_{\beta}}\xi_{\alpha},\xi_{\gamma})$
        \item[(f')] $g([\xi_{\alpha},\xi_{\beta}],X)= \rho_{\alpha\beta}(I_{\beta}X,\xi_{\beta})= \rho_{\alpha\beta}(I_{\alpha}X,\xi_{\alpha})$
    \end{itemize}
    Along the curve we have the following simplifications
    \begin{itemize}
        \item[(a)] $g([X,Y],\xi_{\alpha})=-d\eta_{\alpha}(X,Y)=-2g(I_{\alpha}X,Y)$
        \item[(b)] $g([X,Y],Z)=0$
        \item[(c)] $g([\xi_{\alpha},X],\xi_{\beta})=0$
        \item[(d)] $g([\xi_{\alpha},X],Y)=  -T(\xi_{\alpha},X,Y)$
        \item[(e)]  $g([\xi_{\alpha},\xi_{\beta}],\xi_{\gamma})=-d\eta_{\gamma}(\xi_{\alpha},\xi_{\beta})=Sg(\xi_{\alpha\beta},\xi_{\gamma})$
        \item[(f)] $g([\xi_{\alpha},\xi_{\beta}],X)= \rho_{\alpha\beta}(I_{\beta}X,\xi_{\beta})$
    \end{itemize}
\end{lemma}
\begin{proof}
 It is obtained by a direct computation by combining the definition and the properties of the torsion of the Biquard connection \eqref{sixtyfour}
and using the choice of Fermi frame.
\end{proof}
As a consequence of the previous identities we compute the following Poisson brackets of momentum functions.
\begin{lemma}[Poisson brackets] The momentum functions $u_i,v_\alpha$ have the following properties:
    \begin{itemize}
        \item[(a')] $\{v_\alpha,u_{i}\}= d\eta_{\alpha}(\xi_{\tau},X_{i}) v_\tau + \left(  -T(\xi_{\alpha},X_{i},X_{k})+g(\nabla_{\xi_{\alpha}}X_{i},X_{k}) \right)u_k$,
        \item[(b')] $\{v_\alpha,v_\beta\} = -d\eta_{\tau}(\xi_{\alpha},\xi_{\beta})v_{\tau}+ \rho_{\alpha\beta}(I_{\beta}X_{k},\xi_{\beta})u_{k}$,
        \item[(c')] $\{u_i,u_j\} =- 2 g(I_\tau X_i,X_j)v_\tau +g( X_k,[X_i,X_j]) u_k $.
    \end{itemize}
    Moreover, when evaluated along the extremal $\lambda(t)$, one has
    \begin{itemize}
        \item[(a)] $\{v_\alpha,u_{i}\}=  -T(\xi_{\alpha},X_{i},\dot \gamma)$,
        \item[(b)] $\{v_\alpha,v_\beta\} = - d\eta_{\tau}(\xi_{\alpha},\xi_{\beta})v_{\tau}
        + \rho_{\alpha\beta}(I_{\beta}\dot \gamma,\xi_{\beta})=Sg(\xi_{\alpha\beta},\xi_{\tau})v_{\tau}
        + \rho_{\alpha\beta}(I_{\beta}\dot \gamma,\xi_{\beta}) $,
        \item[(c)] $\{u_i,u_j\} =- 2 g(I_\tau X_i,X_j)v_\tau  $.
        \item[(d)] $\partial_{u_k} \{u_i,u_j\}=0$,
        \item[(e)] $\partial_{v_\alpha} \{u_i,u_j\}= -2g(I_\alpha X_i,X_j)$,
    \end{itemize}
\end{lemma}
\begin{proof} These formulas comes as a direct consequence of Lemma~\ref{l:commutatori}, thanks to the following observation: let $X,Y$ be  two smooth vector fields on $M$ and $h_{X}(\lam)=\langle \lambda, X(x)\rangle$ and $h_{Y}(\lam)=\langle \lam,Y(x)\rangle$ the associated Hamiltonians that are linear on fibers (here $x=\pi(\lam)$). Then we have the identity $\{h_{X},h_{Y}\}(\lam)=\langle \lambda, [X,Y](x)\rangle=:h_{[X,Y]}$.
 
Let us now prove, as an example, formula (c'). We have, by definition
$$[X_{i},X_{j}]= g( X_k,[X_i,X_j])X_{k} +g( \xi_\alpha,[X_i,X_j]) \xi_{\alpha}.$$
Using then the above observation one has
$$\{u_i,u_j\}  = v_\alpha g( \xi_\alpha,[X_i,X_j])  + u_k  g( X_k,[X_i,X_j]).$$
and using then Lemma~\ref{l:commutatori}, one gets
\begin{align*}
   \{u_i,u_j\} & = v_\alpha g( \xi_\alpha,[X_i,X_j])  + u_k  g( X_k,[X_i,X_j])   \\
       & = -v_\alpha d\eta_\alpha(X_i,X_j) + u_k  g( X_k,[X_i,X_j])  = -2v_\alpha g(I_\alpha X_i, X_j) + u_k  g( X_k,[X_i,X_j]) .
\end{align*}
which proves (c'). Observe that the last term vanishes when evaluated along the extremal, thanks to the properties of Fermi frame. This proves (c). All other formulas follow analogously.
\end{proof}

\begin{lemma}[Some arrows]
We have the following expressions along the extremal
\begin{itemize}
\item[(a)] $\overrightarrow{\{u_i,u_j\}} = -2g(I_\alpha X_i,X_j)\vec{v}_\alpha -u_k X_\ell g([X_i,X_j],X_{k}) \partial_{u_\ell} -u_k \xi_\beta g( [X_i,X_j],X_{k}) \partial_{v_\beta}$,
\item[(b)] $\overrightarrow{\{v_{\alpha},u_{i}\}}= -T(\xi_{\alpha},X_{i},X_{k})\vec{u}_k - K_{\alpha i}^{\ell} \partial_{u_{\ell}} -J_{\alpha i}^{\beta} \partial_{v_{\beta}}$,
\end{itemize}
where we set
\begin{align}\label{kj1} K_{\alpha i}^{\ell}:=v_{\tau}X_{\ell}d\eta_{\alpha}(\xi_{\tau},X_{i})  -
u_{k}X_{\ell}T(\xi_{\alpha},X_{i},X_{k})
+u_kX_{\ell}g(\nabla_{\xi_{\alpha}}X_{i},X_{k}),\\\nn \label{kj2}
J_{\alpha
i}^{\beta}:=v_{\tau}\xi_{\beta}d\eta_{\alpha}(\xi_{\tau},X_{i})  -
u_{k}\xi_{\beta}T(\xi_{\alpha},X_{i},X_{k})
+u_k\xi_{\beta}g(\nabla_{\xi_{\alpha}}X_{i},X_{k}).
\end{align}
\end{lemma}
\begin{proof} In this proof we make use of the following two facts. Let $f$ be a smooth function on $T^{*}M$ and let $\{\partial_{u_i},\partial_{v_\alpha},\vec{u}_i,\vec{v}_\alpha\}$ be the Hamiltonian frame, one has
$$\vec f=-X_{i}(f)\partial_{u_{i}}-\xi_{\alpha}(f)\partial_{v_{\alpha}}+\partial_{u_{i}}(f) \vec{u}_{i}+\partial_{v_{\alpha}}(f) \vec{v}_{\alpha}.$$
Moreover, if $f,g$ are smooth functions on $T^{*}M$, we have $\overrightarrow{fg}=g\vec f  +f\vec g$. To prove (a), one then computes
\begin{align*}
\overrightarrow{\{u_i,u_j\}} & = -2 g(I_\alpha X_i,X_j) \vec{v}_\alpha +\vec{u_k}\cancel{ g( X_k,[X_i,X_j]) } - 2v_\alpha \overrightarrow{g(I_\alpha X_i,X_j)} +u_k\overrightarrow{  g( X_k,[X_i,X_j]) }\\
& = -2g(\phi_\alpha X_i,X_j)\vec{v_\alpha} +2v_\alpha \cancel{X_\ell g(I_\alpha X_i,X_j)}\partial_{u_\ell} +2v_\alpha \cancel{\xi_\tau g(I_\alpha X_i,X_j)}\partial_{v_\tau} \\
& \quad -u_k X_\ell g( X_k,[X_i,X_j]) \partial_{u_\ell} -u_k \xi_\tau g( X_k,[X_i,X_j]) \partial_{v_\tau},
\end{align*}
where the barred terms vanishes by Fermi frame. Similarly for (b) one gets
\begin{align*}
\overrightarrow{\{v_\alpha,u_i\}} & =\overrightarrow{d\eta_{\alpha}(\xi_{\tau},X_{i}) v_\tau +
\left(  -T(\xi_{\alpha},X_{i},X_{k})+g(\nabla_{\xi_{\alpha}}X_{i},X_{k}) \right)u_k} \\
&= \cancel{d\eta_{\alpha}(\xi_{\tau},X_{i})} \vec{v}_\tau + \left(  -T(\xi_{\alpha},X_{i},X_{k})+
\cancel{g(\nabla_{\xi_{\alpha}}X_{i},X_{k})} \right)\vec{u}_k \\
&\qquad + \overrightarrow{d\eta_{\alpha}(\xi_{\tau},X_{i})} v_\tau +
\left( \overrightarrow { -T(\xi_{\alpha},X_{i},X_{k})+g(\nabla_{\xi_{\alpha}}X_{i},X_{k})} \right)u_k\\
&=  -T(\xi_{\alpha},X_{i},X_{k})\vec{u}_k- K^{\ell}_{\alpha i}
\partial_{u_{\ell}} -J^{\beta}_{\alpha i} \partial_{v_{\beta}}
\end{align*}
again the barred terms vanishes by Fermi frame.
The expression of the coefficients $K^{\ell}_{\alpha i}$ and $J^{\beta}_{\alpha i}$ are obtained from  direct computations.
\end{proof}
\begin{lemma}\label{l:secder} Let  $v_\alpha(t) =  \langle \lambda(t),\xi_\alpha|_{\gamma(t)}\rangle $, for $\alpha =1,2,3$. Then, along the geodesic, we have
    \begin{equation}\label{geodesic}
    \nabla_{\dot\gamma}\dot\gamma = -2 v_\alpha I_\alpha \dot\gamma.
    \end{equation}
\end{lemma}
\begin{proof}
    Indeed $\gamma(t) = u_i(t) X_i|_{\gamma(t)}$, with $u_i(t) = \langle \lambda(t),X_i|_{\gamma(t)}\rangle $. Then, suppressing the explicit dependence on $t$ one has
    \begin{align*}
    \nabla_{\dot\gamma}\dot\gamma & = \dot{u}_i X_i + u_i u_k \cancel{\nabla_{X_k} X_i}  = \{H,u_i\} X_i  \\
    & = u_k \{u_k ,u_i\} X_i  = -2u_k v_\alpha  g( I_\alpha X_k,X_i)  X_i =-2 v_\alpha  g( I_\alpha \dot\gamma,X_i)  X_i  =-2v_\alpha I_\alpha \dot\gamma,
    \end{align*}
where the barred term vanishes along the trajectory thanks to the properties of Fermi frame.
\end{proof}
\subsection{Fundamental computations}
The frame $\{\partial_{u},\partial_{v},\vec u, \vec v\}$ is a basis of the tangent space to $T^{*}M$. We compute the differential equations of this frame along an extremal.

\begin{lemma}\label{l:fund} Along the extremal $\lambda(t)$, we have
    \begin{align*}
    \dot{\partial}_v & = 2A\partial_u, \\
    \dot{\partial}_u & = - \vec{u} + G \partial_{v},\\
    \dot{\vec{u}} & = 2C\vec{u}-2A^*\vec{v}+B\partial_u+D\partial_v, \\
    \dot{\vec{v}} & = L\vec{u}+M\partial_u+2N\partial_v,
    \end{align*}
    where we defined the following matrices, computed along $\lambda(t)$:
    \begin{align*}
    A_{\beta i} & := g( I_\beta\dot{\gamma},X_i) ,   \quad  \hspace{80mm} \text{$3 \times 4n$ matrix}, \\
    G_{i\alpha} & :=   -T(\xi_{\alpha},\dot \gamma,X_{i}) 
\hspace{80mm} \text{ $4n\times 3$  matrix}, \\
    B_{i\ell} &:=-u_{j}u_{k}X_{\ell}g([X_{j},X_{i}],X_{k})=R(\dot\gamma, X_{\ell},X_{i},\dot \gamma)\\
    C_{ij} & := v_\alpha g( I_\alpha X_i,X_j) , \quad \hspace{43mm}  \text{ $4n\times 4n$ skew-symmetric matrix}, \\
    D_{i\beta} &  :=  -u_{j}u_{k}\xi_{\beta}g([X_{j},X_{i}],X_{k})=R(\dot\gamma, \xi_{\beta},X_{i},\dot \gamma)\\
    L_{\beta j} &:=-\frac12 \left( T^{0}(I_{\beta}X_{j},\dot \gamma)+T^{0}(I_{\beta}\dot \gamma,X_{j})\right)\\
    M_{\alpha\ell} &:=K_{\alpha j}^{\ell}u_{j}, \quad
    M_{\alpha\ell}=-2v_{\tau}\rho_{\zeta}(X_\ell,\dot\gamma)+\frac12(\nabla_{X_{\ell}}T^0)(I_{\alpha}\dot\gamma,\dot\gamma), \\
    N_{\alpha\beta} &:=J_{\alpha j}^{\beta}u_{j}, \quad 2 N_{\alpha\beta}= -2v_{\tau}\rho_{\zeta}(\xi_{\beta},\dot\gamma)+
\frac12(\nabla_{\xi_{\beta}}T^0)(I_{\alpha}\dot\gamma,\dot\gamma).
    \end{align*}
in the last two formulas $\{\alpha \tau \zeta\}$ is a cyclic permutation of $\{1 2 3\}$,  or $\zeta=\alpha\tau$ (as product of quaternions).
\end{lemma}

\begin{proof}
    By a direct computation we get (simplifications are due to Fermi frame properties)
\begin{align*}
    \dot{\partial}_{v_\beta} &=
    [u_j\vec{u}_j,\partial_{v_\beta}] =  -\partial_{v_\beta}(u_j)\vec{u}_j + u_j[\vec{u}_j,\partial_{v_\beta}] = u_j[\vec{u}_j,\partial_{v_\beta}](u_i)\partial_{u_i} + u_j[\vec{u}_j,\partial_{v_\beta}](v_\alpha)\partial_{v_\alpha}\\
    &= -u_j\partial_{v_\beta}\{u_j,u_i\}\partial_{u_i} -u_j \cancel{\partial_{v_\beta}\{u_j,v_\alpha\}} \partial_{v_\alpha}= u_j g( 2I_\beta X_j,X_i) \partial_{u_i} = 2 g( I_\beta\dot\gamma,X_i) \partial_{u_i}.\\
    \dot{\partial}_{u_i} &=
    [u_j\vec{u}_j,\partial_{u_i}] = -\partial_{u_i}(u_j)\vec{u}_j + u_j[\vec{u}_j,\partial_{u_i}] = -\vec{u}_i -u_j \cancel{\partial_{u_i}\{u_j,u_\ell\}}\partial_{u_\ell} - u_{j} \partial_{u_i}\{u_j,v_\alpha\}\partial_{v_\alpha}\\
    &=-\vec{u}_i  -u_{j} T(\xi_{\alpha},X_j,X_k)\partial_{u_i}(u_k)\partial_{v_\alpha}=-\vec{u}_i  -T(\xi_{\alpha},\dot{\gamma},X_i)\partial_{v_\alpha}\\
    \dot{\vec{u}}_i &=
    [u_j\vec{u}_j,\vec{u}_i] = -\vec{u}_i(u_j)\vec{u}_j + u_j[\vec{u}_j,\vec{u}_i] = -\{u_i,u_j\} \vec{u}_j + u_j\overrightarrow{\{u_j,u_i\}}\\
 &= 2v_\alpha g(I_\alpha X_i,X_j )  \vec{u}_j +2g(I_\alpha X_i,\dot \gamma)\vec{v}_\alpha +u_{j}u_k X_\ell g([X_i,X_j],X_{k}) \partial_{u_\ell}  +u_{j}u_k \xi_\beta g( [X_i,X_j],X_{k}) \partial_{v_\beta}.\\
 \dot{\vec{v}}_\beta &=
    [u_j\vec{u}_j,\vec{v}_\beta] = -\vec{v}_\beta(u_j)\vec{u}_j + u_j[\vec{u}_j,\vec{v}_\beta] = -\{v_\beta,u_j\}\vec{u}_j + u_j\overrightarrow{\{u_j,v_\beta\}} \\&= T(\xi_{\beta},X_{j},X_{k})u_{k} \vec{u}_{j}+u_j\overrightarrow{\{u_j,v_\beta\}}\\
    &=  T(\xi_{\beta},X_{j},X_{k})u_{k} \vec{u}_{j}-u_{j}[  -T(\xi_{\alpha},X_{j},X_{k})\vec{u}_k - K_{\alpha j}^{\ell} \partial_{u_{\ell}} -J_{\alpha j}^{\beta} \partial_{v_{\beta}}]\\
    &= [T(\xi_{\beta},X_{j},\dot \gamma)+T(\xi_{\beta},\dot \gamma,X_{j})] \vec{u}_{j} + K_{\alpha j}^{\ell}u_{j} \partial_{u_{\ell}} +J_{\alpha j}^{\beta}u_{j} \partial_{v_{\beta}}\qedhere
\end{align*}
\end{proof}
The previous proof is completed thanks to the next lemma.
\begin{lemma}\label{l7} In terms of Biquard curvature we have
\begin{itemize}
\item[(a)] $B_{i\ell}
=-u_{j}u_{k}X_{\ell}g([X_{j},X_{i}],X_{k})=R(\dot\gamma,
X_{\ell},X_{i},\dot \gamma)$
\item[(b)] $D_{i\beta} =
-u_{j}u_{k}\xi_{\beta}g([X_{j},X_{i}],X_{k})=R(\dot\gamma,\xi_{\beta},X_{i},\dot
\gamma)$  %given by \eqref{vert}
\item[(c)]
$M_{\alpha
\ell}=-2v_{\tau}\rho_{\zeta}(X_l,\dot\gamma)+\frac12(\nabla_{X_{\ell}}T^0)(I_{\alpha}\dot\gamma,\dot\gamma)$,
\item[(d)]
$2N_{\alpha\beta}=-2v_{\tau}\rho_{\zeta}(\xi_{\beta},\dot\gamma)+
\frac12(\nabla_{\xi_{\beta}}T^0)(I_{\alpha}\dot\gamma,\dot\gamma)$,
\end{itemize}
where in the last two formulas $\{\alpha \tau \zeta\}$ is a cyclic permutation of $\{1 2 3\}$.
\end{lemma}
\begin{proof}[Proof of Lemma \ref{l7}] We use classical tricks in curvature calculations
\begin{align*}
    R(\dot\gamma,X_i,X_\ell,\dot\gamma) & = u_k u_j g(\nabla_{X_k}\nabla_{X_i}X_\ell - \nabla_{X_i}\nabla_{X_k}X_\ell - \nabla_{[X_k,X_i]}X_\ell , X_j) \\
    & =  \bcancel{u_k u_j X_k g(\nabla_{X_i}X_\ell,X_j)}  - \cancel{u_k u_j g(\nabla_{X_i}X_\ell,\nabla_{X_k}X_j)}  - u_k u_j X_i g(\nabla_{X_k}X_\ell,X_j) \\
    &\quad  + u_k u_j \cancel{g(\nabla_{X_k}X_\ell,\nabla_{X_i}X_j) } - u_k u_j g(\cancel{\nabla_{[X_k,X_i]}X_\ell}, X_j) .
    \end{align*}
    where the cancellation \cancel{XX} follows from properties of Fermi frame, while \bcancel{YY} is due to the identity $u_k X_k g(\nabla_{X_i}X_\ell,X_i) =0$ (notice that the last identity is the derivative in the direction of $\dot\gamma(t)$ of the following one $ g(\nabla_{X_i}X_\ell,X_j) |_{\gamma(t)} = 0$). Thus using that the torsion among horizontal vector fields  is vertical
\begin{align*}
    R(\dot\gamma,X_i,X_\ell,\dot\gamma)    & =    - u_k u_j X_i g(\nabla_{X_k}X_\ell,X_j)\\
    &=   -u_k u_j X_i g(\nabla_{X_\ell}X_k,X_j)-u_k u_j X_i g([X_{k},X_{\ell}],X_j).
    \end{align*}
On the other hand, the term $g(\nabla_{X_\ell}X_k,X_j)$ is skew-symmetric in $k,j$ and we have a symmetric sum so the first term is zero and
\begin{align*}
    R(\dot\gamma,X_i,X_\ell,\dot\gamma)=-u_k u_j X_i g([X_{k},X_{\ell}],X_j).
    \end{align*}
Along the same path one can show that
    \begin{align*}
    R(\dot\gamma,\xi_\beta,X_\ell,\dot\gamma)=-u_k u_j \xi_\beta
    g([X_{k},X_{\ell}],X_j).
    \end{align*}
Applying \eqref{d3n5}, \eqref{d3n6},  we have
\begin{align*}%\label{vert}
    R(\dot\gamma,\xi_\beta,X_\ell,\dot\gamma)&=(\nabla_{\dot{\gamma}}U)(I_{\beta}X_{\ell},\dot{\gamma})-
\frac14(\nabla_{\dot{\gamma}}T^0)(I_{\beta}X_{\ell},\dot{\gamma})\\
& -\frac14(\nabla_{\dot{\gamma}}T^0)(X_{\ell},I_{\beta}\dot{\gamma})
+\frac12(\nabla_{X_{\ell}}T^0)(I_{\beta}\dot{\gamma},\dot{\gamma})\\
&-2g(I_{\tau}\dot{\gamma},X_{\ell})\rho_{\zeta}(I_{\beta}\dot{\gamma},\xi_{\beta})
+2g(I_{\zeta}\dot{\gamma},X_{\ell})\rho_{\tau}(I_{\beta}\dot{\gamma},\xi_{\beta}),
 \end{align*}
where $\{\beta \tau \zeta\}$ is a cyclic permutation of $\{1 2 3\}$. Further, we have using \eqref{kj1}
   \begin{align*}%\label{kj11}
   M_{\alpha\ell} =K_{\alpha i}^{\ell}u_{i}&=u_{i}v_{\tau}X_{\ell}d\eta_{\alpha}(\xi_{\tau},X_{i})
   -u_{i}u_{k}X_{\ell}T(\xi_{\alpha},X_{i},X_{k})
+u_{i}u_kX_{\ell}g(\nabla_{\xi_{\alpha}}X_{i},X_{k})\\
&=-u_iv_{\tau}X_{\ell}g(\nabla_{X_i}\xi_{\alpha},\xi_{\tau})+
\frac14u_{i}u_k(\nabla_{X_{\ell}}T^0)[(I_{\alpha}X_i,X_k)+(I_{\alpha}X_k,X_i)]\\
&=-2v_{\tau}\rho_{\zeta}(X_l,\dot\gamma)+\frac12(\nabla_{X_{\ell}}T^0)(I_{\alpha}\dot\gamma,\dot\gamma),
\end{align*}
where $\{\alpha \tau \zeta\}$ is a cyclic permutation of $\{1 2
3\}$, the skew-symmetric parts in $i$ and $k$ of the first line
are cancelled and the first term of the second line is evaluated
as follows
\begin{align*}
u_iX_{\ell}g(\nabla_{X_i}\xi_{\alpha},\xi_{\tau})=u_ig(\nabla_{X_{\ell}}\nabla_{X_i}\xi_{\alpha},\xi_{\tau})+
u_i\cancel{g(\nabla_{X_i}\xi_{\alpha},\nabla_{X_{\ell}}\xi_{\tau})}\\=u_ig(\nabla_{X_i}\nabla_{X_{\ell}}\xi_{\alpha},\xi_{\tau})+
u_iR(X_{\ell},X_i,\xi_{\alpha},\xi_{\tau})+\cancel{u_ig(\nabla_{[X_{\ell},X_i]}\xi_{\alpha},\xi_{\tau})}\\=
2\rho_{\zeta}(X_{\ell},\dot{\gamma})+\bcancel{u_iX_ig(\nabla_{X_{\ell}}\xi_{\alpha},\xi_{\tau})}
\end{align*}
Similarly one gets the formula for $N_{\alpha\beta}$
\end{proof}

\section{Symplectic products and canonical frame} \label{s:5}

In this section, for $n$-tuples
$v,w$ of vectors, the symbol
$\sigma(v,w)$ denotes the matrix $$\sigma(v,w):=(\sigma(v_i,w_j))_{i,j=1,\ldots,n},$$ 
whose entries are symplectic products. Notice that with this convention one has the identities
$$\sigma(v,w)^* = -\sigma(w,v), \qquad \sigma(Av,Bw)=A\sigma(v,w) B^{*}.$$ 
 where, for $n$-tuple of
vectors $v$  and a matrix $L$, the juxtaposition $Lv$ denotes the
$n$-tuple of vectors obtained by matrix multiplication.  
When $v,w$ are vector fields defined along an extremal $\lam(t)$, the following Leibniz rule holds
$$\frac{d}{dt} \sigma(v,w)
= \sigma(\dot{v},w) + \sigma(v,\dot{w}).$$

\subsection{Derivatives of $\partial_{v}$ and symplectic products of the coordinate frame}
By a direct computation one gets from Lemma \ref{l:fund} the following relations.
\begin{lemma}\label{l:dots} Along the extremal, we have
    \begin{align}\nn
    \dot\partial_v & = 2 A \partial_u, \\ \nn %\label{ddotv}
    \ddot\partial_v & = 2 \dot{A} \partial_u +2AG\partial_{v}- 2 A\vec{u}, \\ \nn
    %\dddot{\partial_v} & =2\dot A G \partial_v +(2\ddot A +4AGA +B)\partial_u +4\dot A A^{*}  \vec{v} +(-2\dot A-4 \dot AC)\vec{u}\\
    \dddot{\partial_v} & = (4\dot A G+2A\dot G-2AD)\partial_v+(2\ddot A +4AGA -2AB)\partial_u+4 A A^{*}  \vec{v}+(-4\dot A-4 AC)\vec{u}\\ \nn %\label{eq:31} 
&=(4\dot A G+2A\dot G-2AD)\partial_v+(2\ddot A +4AGA -2AB)\partial_u+4 \vec{v} -4(3VA+v\dot{\gamma}^*)\vec{u}
    %48V^2 \partial_v -2\left[8VA(1-\|v\|^2-B) + A\dot{B} + 8\|v\|^2v \dot\gamma^* \right] \partial_u.
    \end{align}
\end{lemma}
 To prove the fourth equality  we used  \eqref{idav} and \eqref{dota} below.
We then  compute symplectic products of the elements of the basis. 
\begin{lemma}\label{sympl}
The non-zero brackets between $\partial_{u_i},\partial_{v_\alpha},\vec{u}_i,\vec{v}_\alpha$ are
\begin{itemize}
\item[(a)]  $\sigma(\partial_u,\vec{u}) = \mathbbold{1}$,
\item[(b)] $ \sigma(\partial_v,\vec{v}) = \mathbbold{1}$,
\item[(c)] $\sigma(\vec{u},\vec{u}) =\{u_i,u_j\} = -2 C$,
\item[(d)] $\sigma(\vec{v},\vec{v}) =\{v_{\alpha},v_{\beta}\}=SV +\chi$,
\item[(e)]  $\sigma(\vec{u},\vec{v})=\{u_i,v_{\alpha}\}=P$,
\end{itemize}
where, according to \eqref{sixtyfour},  $$\chi_{\alpha\beta}=- \rho_{\alpha\beta}(I_{\beta}\dot\gamma,\xi_{\beta})=\rho_{\alpha\beta}(\xi_{\beta},I_{\beta}\dot\gamma)=-g([\xi_{\alpha},\xi_{\beta}],\dot{\gamma})=-\chi_{\beta\alpha},$$
 and, thanks to \eqref{need1},
 $$P_{i\alpha}=\{u_i,v_{\alpha}\}= T(\xi_{\alpha},X_i,\dot{\gamma})=T(\xi_{\alpha},\dot{\gamma},X_i)+2U(I_{\alpha}X_i,\dot{\gamma})=-G_{i\alpha}+2U(I_{\alpha}X_i,\dot{\gamma}).$$
\end{lemma}
Notice that, by bilinearity of $\sigma$, Lemma~\ref{sympl} permits to compute the symplectic product of any pair of vectors.

\begin{lemma}[Several identities]\label{idens} We have the following identities
\begin{gather}
AA^*=\mathbbold{1}, \quad A\dot{\gamma}=0, \quad Vv=0, \quad AC=VA-v\dot{\gamma}^* ,  \quad
V^2=vv^*-||v||^2\mathbbold{1},\label{idav}\\
\ddot{\gamma}=-2A^*v=2C\dot{\gamma},  \label{secder}\\
\dot A%=-2v_{\alpha}g(I_{\beta}I_{\alpha}\dot{\gamma},X_i)
=2VA+2v\dot{\gamma}^*, \quad \ddot{A}=2\dot{V}A+2\dot{v}\dot{\gamma}^*-4\|v\|^2A. \label{dota}
\end{gather}
\end{lemma}
\begin{proof}The identities \eqref{idav} follow directly from the definitions while \eqref{secder} is precisely \eqref{geodesic} written in terms of $A,C$ and $v$.
For \eqref{dota},
working in the Fermi
frame along $\gamma$, we have
\begin{align*}%\label{dot1}
\dot A_{\beta i}&=\{H,A_{\beta i}\}=u_j\vec{u_j}(A_{\beta i})=-u_j\vec{A_{\beta i}}(u_j)\\
&=u_jX_l(g(I_{\beta}\dot{\gamma},X_i))\partial_{{u_l}}(u_j)+u_j\xi_{\tau}(g(I_{\beta}\dot{\gamma},X_i))\partial_{{v_{\tau}}}(u_j)\\
&=u_jX_j(g(I_{\beta}\dot{\gamma},X_i))=g(I_{\beta}\nabla_{\dot{\gamma}}\dot{\gamma},X_i)=-2v_{\alpha}g(I_{\beta}I_{\alpha}\dot{\gamma},X_i) =2V_{\beta\alpha}A_{\alpha i}
+2v\dot{\gamma}^*\,
\end{align*}
where we used \eqref{geodesic} to get the last equality.
\end{proof}
\begin{corollary}\label{idens1} We also have
\begin{gather}\label{morid}
\dot{A}A^*=2V,\quad A\dot{A}^*=-2V, \quad ACA^*=V, \quad  \dot{A}+AC=3VA+v\dot{\gamma}^*, \quad V^3=-||v||^2V, \\\label{iddota} \ddot{A}=2\dot{V}A+2\dot{v}\dot{\gamma}^*-4\|v\|^2A, \quad  A\ddot{A}^*=-2\dot{V}-4\|v\|^2\mathbbold{1}
\end{gather}
\end{corollary}
\begin{proof}
The identities \eqref{morid} follow directly from \eqref{idav} and \eqref{dota}.  For the first one in \eqref{iddota}, we take the derivative of \eqref{dota}  applying \eqref{idav} and \eqref{secder} to get
that
\begin{multline*}
 \ddot{A}=2\dot{V}A+2V\dot{A}+2\dot{v}\dot{\gamma}^*+2v\ddot{\gamma}^*=2\dot{V}A+2V(2VA+2v\dot{\gamma}^*)+2\dot{v}\dot{\gamma}^*-4vv^*A\\
=2\dot{V}A+4(vv^*-||v||^2\mathbbold{1})A+2\dot{v}\dot{\gamma}^*-4vv^*A=2\dot{V}A+2\dot{v}\dot{\gamma}^*-4\|v\|^2A.
\end{multline*}
The second identity in \eqref{iddota} follows from the first one applying \eqref{idav}.
\end{proof}

\begin{lemma}[Derivative of $V$]\label{dtv}
We have $\dot{V}\not=0$. In particular
\begin{equation}\label{dotv}
\dot{v}_{\tau}= T(\xi_{\tau},\dot{\gamma},\dot{\gamma}).
\end{equation}
In vector notation $\dot v=-G^*\dot{\gamma}$. In particular  
if $T^0=0$ then $\dot{v}=\dot{V}=0.$
\end{lemma}
\begin{proof}It easily follows by $$
\dot{v}_{\tau}=\{H,v_{\tau}\}=u_{j}\{u_j,v_{\tau}\}=u_jT(\xi_{\tau},X_{j},\dot{\gamma})=T(\xi_{\tau},\dot{\gamma},\dot{\gamma}).\hfill \qedhere$$
\end{proof}

\subsection{Computation of symplectic products}

Now we deduce symplectic products of derivatives of the vector  $\partial_{v}$ using Lemma~\ref{l:dots},  Lemma~\ref{sympl}, Lemma~\ref{idens} and Corollary~\ref{idens1}.

\begin{lemma} \label{l:pro2}
We have
\begin{gather}
\sigma(\partial_v,\partial_v) = 0, \quad  \sigma(\partial_v,\dot \partial_v) = 0, \quad \sigma(\partial_v,\ddot \partial_v) = 0, \quad \sigma(\partial_v,\dddot \partial_v)=4\mathbbold{1}, \label{sigma0}\\ \sigma(\dot\partial_v,\dot \partial_v)=0,\quad
\sigma(\dot\partial_v,\ddot \partial_v)=-4\mathbbold{1}, \quad  \sigma(\dot \partial_v,\dddot \partial_v)=24V,\quad \sigma( \ddot \partial_v, \ddot \partial_v)= -24V,  \label{sigma1}
%\\
%  \sigma(\ddot \partial_{v},\dddot \partial_{v})=96V^2-8\dot{V}-16||v||^2\mathbbold{1}+8AG-8AP+8G^*A^*-4AB^*A^*.
%  \label{sigma2}
\end{gather}
\end{lemma}
\begin{proof}
The identities \eqref{sigma0} and the first two ones in \eqref{sigma1} follow from Lemma~\ref{l:dots},  Lemma~\ref{sympl} and  Lemma~\ref{idens}. We calculate
$ \sigma(\dot \partial_v,\dddot \partial_v)=2A(-12VA-4v\dot{\gamma}^*)^* =-24AA^*V^*-8A\dot{\gamma}v^*=24V$ which proves the third one in \eqref{sigma1}.
Differentiating the second identity in \eqref{sigma1}, one gets $0=\sigma( \ddot \partial_v, \ddot \partial_v)+\sigma( \dot \partial_v, \dddot \partial_v)=\sigma( \ddot \partial_v, \ddot \partial_v)+24V$ which yields the third one in \eqref{sigma1}.
\end{proof}

\subsection{Canonical frame} To compute $\mc{R}_{cc}=\sigma(\dot{F}_c,F_c)$, we need to compute the elements of the canonical basis up to $F_{c}$.
The algorithm to recover them (following the general construction developed in \cite{lizel}) starts from identifying $E_{a}$ and then works as follows:
\begin{equation*}
E_a \rightarrow E_b \rightarrow F_b  \rightarrow E_c  \rightarrow F_c \rightarrow \mc{R}_{cc}
\end{equation*}
The triplet $E_a$ is determined by the following four conditions:
\begin{itemize}
    \item[(i)]$\pi_*E_a=0$,
    \item[(ii)]$\pi_*\dot{E}_a=0$,
    \item[(iii)] $\sigma(\ddot{E}_a,\dot{E}_a)=\mathbbold{1}$,
    \item[(iv)] $\sigma(\ddot{E}_a,\ddot{E}_a) =\mathbbold{0}$.
\end{itemize}
Items (i) and (ii) imply that there exists $M \in \mathrm{GL}(3)$ such that $E_a = M \partial_v$. Condition (iii) implies that $M = \tfrac{1}{2} O$ with $O \in \mathrm{O}(3)$. Finally, (iv) implies that $O$ satisfies the differential equation
    \begin{equation}\label{doto}
    \dot{O} = \frac{1}{16} O \sigma(\ddot{\partial}_{v_\alpha},\ddot{\partial}_{v_\beta}) =  -\frac{3}{2}OV.
    \end{equation}
Its solution is unique up to an orthogonal transformation (the initial condition, that we set $O(0) = \mathbbold{1}$).
Using the structural equations  together with \eqref{doto}, we have
\begin{align}\nn
E_a & = \tfrac{1}{2} O \partial_{v},\\ \nn
E_b & = \dot{E}_a =\tfrac{1}{2}O (-\frac32V\partial_v + \dot\partial_v) , \\%\label{eb}\\
F_b & = -\dot{E}_b = % - \tfrac{1}{2} O[(\V^2+\dot{\V}) \partial_v +2 \V \dot\partial_v + \ddot{\partial}_v] \\  &=
-  \tfrac{1}{2} O[(\frac94V^2-\frac32\dot{V}) \partial_v -
3V \dot\partial_v + \ddot{\partial}_v]  \label{fb}
\end{align}
Thus we can also compute \begin{align}
\dot{F}_b  &  =% - \tfrac{1}{2} O[\V^3+3\V\dot{\V}+\ddot{\V}) \partial_v
%+ 3( \V^2+\dot{\V}) \dot\partial_v + 3 \V \ddot{\partial}_v +\dddot{\partial}_v],
%\\\label{dfb} &=
- \tfrac{1}{2} O[((-\frac{27}8V^3+\frac{27}4V\dot{V}-\frac32\ddot{V}) \partial_v
+3( \frac94V^2-\frac32\dot{V}) \dot\partial_v -\frac92V \ddot{\partial}_v +\dddot{\partial}_v].\label{dfb}
\end{align}

\medskip

The next step is to compute $E_c$. It is determined by the following conditions:
   \begin{itemize}
       \item[(i)] $\pi_*E_c=0$,
     \item[(ii)] $\sigma(E_c,F_c)=\mathbbold{1}$ and $\sigma(E_c,F_b)=\sigma(E_c,F_a)=\mathbbold{0}$,
      \item[(iii)] $\pi_*\ddot{E}_c=0$.
  \end{itemize}
For (i) we can write
\begin{equation}\label{ec}E_c=Y\partial_u+W\partial_v, \quad F_c=-\dot{E_c}=-(\dot{Y}+2WA)\partial_u-(YG+\dot{W})\partial_v+Y\vec{u}.
\end{equation}
where $Y$ is a $(4n-3)\times 4n$ matrix and $W$ is a $(4n-3)\times
3$ matrix.

\medskip
To compute $\sigma(E_c,F_c), \sigma(E_c,F_b), \sigma(E_c,\dot{F}_b) $ using \eqref{ec}, \eqref{fb} and \eqref{dfb},  we need to know
$$\sigma(\partial_u,\dot{\partial_v}),\ \sigma(\partial_u,\ddot{\partial_v}),\ \sigma(\partial_u,\dddot{\partial_v}),\ \sigma(\partial_v,\dot{\partial_v}),\ \sigma(\partial_v,\ddot{\partial_v}),\  \sigma(\partial_v,\dddot{\partial_v}).$$
The only non-zero terms using Lemma~\ref{l:dots} are given by:
\begin{gather*}  \sigma(\partial_u,\ddot{\partial_v})=-2A^*, \quad  \sigma(\partial_v,\dddot{\partial_v})=4\mathbbold{1},\\
\sigma(\partial_u,\dddot{\partial_v})=-4(\dot{A}^*+C^*A^*)=-4(3A^*V^*+\dot{\gamma}v^*).
\end{gather*}
For  (ii),  observing that $\sigma(E_c,F_a)=\mathbbold{0}$ implies
$\sigma(E_c,\dot{F}_b)=\mathbbold{0}$,  we get
\begin{equation}\label{uu}
  YY^*  =\mathbbold{1}, \qquad YA^*  =0, \qquad W  = Y(\dot{A}^*+C^*A^*)=Y(3A^*V^*+\dot{\gamma}v^*)=Y\dot{\gamma}v^*.
 \end{equation}
%where $\dot\gamma$ represents, with no risk of confusion, the $4d$ dimensional column vector that represents $\dot\gamma$ in the frame $\{X_i\}$.
We obtain from \eqref{ec} and \eqref{uu}
$$\pi_*\ddot{E}_c=2(\dot{Y}+WA+YC)\vec{u}.$$
Finally, using (iii) and the equality above, we get that $Y$ must
satisfy
 \begin{equation}\label{eq:eqforU}
\dot{Y}  =-W A -YC=-Y(\dot{A}^*A+C^*A^*+C)=-Y(\dot{\gamma}v^*A+C).
\end{equation}
 Using Lemma~\ref{idens}, \eqref{uu}, \eqref{eq:eqforU} and \eqref{dotv}, we have
\begin{align}\label{dotw}
\dot{W} &=
\dot{Y}\dot{\gamma}v^*+Y\ddot{\gamma}v^*+Y\dot{\gamma}\dot{v}^*=-Y\dot{\gamma}v^*A\dot{\gamma}v^*-YC\dot{\gamma}v^*-2YA^*vv^*+Y\dot{\gamma}T(\xi_{\tau},\dot{\gamma},\dot{\gamma})\\\nn &=YA^*vv^*+Y\dot{\gamma}T(\xi_{\tau},\dot{\gamma},\dot{\gamma})=Y\dot{\gamma}T(\xi_{\tau},\dot{\gamma}.\dot{\gamma})
\end{align}
 Observe that $Y$ represents an orthogonal projection on
$\distr \cap
\spn\{I_{\alpha}\dot\gamma,I_{\beta}\dot\gamma,I_{\tau}\dot\gamma\}^\perp$.
Then
\begin{equation*}
   Y^*Y =\mathbbold{1} - A^*A.
\end{equation*}
Substitute \eqref{eq:eqforU} into the second equality of \eqref{ec} to get
 \begin{equation}\label{fccg}F_c=(YC-WA)\partial_u-(YG+\dot{W})\partial_v+Y\vec{u}.\end{equation}
We calculate from \eqref{fccg} using Lemma~\ref{l:fund},
\eqref{uu}, \eqref{eq:eqforU} and \eqref{dotw} that
\begin{multline}\label{ddfcg}
\dot{F}_c=(\dot{Y}C+Y\dot{C}-\dot{W}A-W\dot{A})\partial_u-(\dot{Y}G+Y\dot{G}+\ddot{W})\partial_v+\dot{Y}\vec{u}\\
+(YC-WA)(-\vec{u}+G\partial_v)
-2(YG+\dot{W})A\partial_u+Y(2C\vec{u}-2A^*\vec{v}
+B\partial_u+D\partial_v)\\
=(\dot{Y}C+Y\dot{C}-\dot{W}A-W\dot{A}+YB-2YGA-2\dot{W}A))\partial_u
-(\dot{Y}G+Y\dot{G}+\ddot{W}-YCG+WAG-YD))\partial_v\\
+\cancel{(\dot{Y}+WA-YC+2YC)}\vec{u}-2\cancel{YA^*}\vec{v}\\
=(\dot{Y}C+Y\dot{C}-\dot{W}A-W\dot{A}+YB-2YGA-2\dot{W}A))\partial_u
-(\dot{Y}G+Y\dot{G}+\ddot{W}-YCG+WAG-YD))\partial_v,
\end{multline}
where the cancellations in the fourth line follow from \eqref{eq:eqforU} and \eqref{uu}.

Applying Lemma~\ref{sympl}, Lemma~\ref{idens},  \eqref{uu} and \eqref{eq:eqforU} we obtain from \eqref{fccg} and \eqref{ddfcg}
\begin{align}\label{rccg}
\mc{R}_{cc}&=\sigma(\dot{F}_c,F_c)=(\dot{Y}C+Y\dot{C}-\dot{W}A-W\dot{A}+YB-2YGA-2\dot{W}A)Y^*\\
&=\dot{Y}CY^*+Y\dot{C}Y^*-\dot{W}AY^*-W\dot{A}Y^*+YBY^*-\dot{W}AY^*\nn\\
&=-Y(\dot{\gamma}v^*AC+CC)Y^*+YBY^*-W(2VA+2v\dot{\gamma}^*)Y^*+Y\dot{C}Y^*\nn\\
&=-Y(\dot{\gamma}v^*AC+CC)Y^*+YBY^*-2Y\dot{\gamma}v^*v\dot{\gamma}^*Y^*+Y\dot{C}Y^*\nn\\
&=-Y(\dot{\gamma}v^*AC+CC-B+2\|v\|^2\dot{\gamma}\dot{\gamma}^*-\dot{C})Y^*\nn\\
&=-Y(\dot{\gamma}v^*(VA-v\dot{\gamma}^*)-\|v\|^2\mathbbold{1}-B+2\|v\|^2\dot{\gamma}\dot{\gamma}^* -\dot{C})Y^*\nn\\
&=Y[B+ \dot{C}+\|v\|^2(1-\dot\gamma\dot\gamma^*)]Y^*,\nn
\end{align}
where we used $C_{ij}C_{jk}=-v_{\alpha}v_{\beta}g(I_{\alpha}X_i,I_{\beta}X_k)=-||v||^2\mathbbold{1}$.

\begin{lemma} \label{dotc}(Derivative of $C$ along $\gamma(t)$)
We have
\begin{equation*}
\dot{C}_{ki}=-
g(I_{\tau}X_k,X_i)T(\xi_{\tau},\dot{\gamma},\dot{\gamma})
\end{equation*}
\end{lemma}
\begin{proof} By a direct computation 
\begin{align*}
\dot{C}_{ki}=\{H,C_{ki}\}&=-u_j\vec{C}_{ki}(u_j)=-u_jX_j(v_{\alpha}g(I_{\alpha}X_k,X_i)+u_j\partial_{v_{\tau}}(v_{\alpha}g(I_{\alpha}X_k,X_i)\vec{v_{\tau}}(u_j)\\
&=-v_{\alpha}[g(\nabla_{\dot{\gamma}}I_{\alpha}X_k,X_i)+g(I_{\alpha}X_k,\nabla_{\dot{\gamma}}X_i)]-u_jg(I_{\tau}X_k,X_i)T(\xi_{\tau},X_j,\dot{\gamma})\\&=-
g(I_{\tau}X_k,X_i)T(\xi_{\tau},\dot{\gamma},\dot{\gamma})\qedhere
\end{align*}\end{proof}

Finally we derive  the following formula for the trace of the matrix $\mc{R}_{cc}$.
\begin{lemma} \label{riccic0}
We have
\begin{equation*}\label{riccc0}
  \trace(\mc{R}_{cc})=\mathrm{Ric}(\dot{\gamma},\dot{\gamma})- \sum_{\alpha=1}^3
R(\dot\gamma,I_{\alpha}\dot\gamma,I_{\alpha}\dot\gamma,\dot\gamma)+4(n-1)\|v\|^2.
\end{equation*}
\end{lemma}
\begin{proof} Indeed, we obtain from \eqref{rccg}
\begin{align*}
 \trace(\mc{R}_{cc})&=\trace((B  +\dot{C}+ \|v\|^2(\mathbbold{1}-\dot\gamma\dot\gamma^*))Y^*Y)\\
 & = \trace((B + \dot{C}+ \|v\|^2(\mathbbold{1}-\dot\gamma\dot\gamma^*))(\mathbbold{1}-A^*A))\\
% & = \trace(B) - \trace(ABA^*) + \|v\|^2\trace(\mathbbold{1}-\dot\gamma\dot\gamma^*) - \|v\|^2\trace(A^*A)\\
 & = \sum_{i=1}^{4n}R(\dot\gamma,X_i,X_i,\dot\gamma) - \sum_{\alpha=1}^{3} R(\dot\gamma,I_{\alpha}\dot\gamma,I_{\alpha}\dot\gamma,\dot\gamma)+(4n-4)\|v\|^2\\
& =\mathrm{Ric}(\dot{\gamma},\dot{\gamma})- \sum_{\alpha=1}^3
R(\dot\gamma,I_{\alpha}\dot\gamma,I_{\alpha}\dot\gamma,\dot\gamma)+4(n-1)\|v\|^2.%\label{riccc}
\end{align*}
where we used $\mathrm{tr}(\dot{C}\mathbbold{1})=0$ and
\begin{align*}\mathrm{tr}(\dot{C}A^*A)=\dot{C}_{ki}A_{i\beta}A_{\beta k}&=-T(\xi_{\tau},\dot{\gamma},\dot{\gamma})g(I_{\tau}X_k,X_i)g(I_{\beta}\dot{\gamma},X_i)g(I_{\beta}\dot{\gamma},X_k)\\&=
T(\xi_{\tau},\dot{\gamma},\dot{\gamma})g(I_{\beta}\dot{\gamma},I_{\tau}I_{\beta}\dot{\gamma})=0. \qedhere
\end{align*}
\end{proof}
\section{Proof of Theorems \ref{t:main1} and \ref{t:main2}}\label{s:proofmain}

Theorem \ref{t:main1} is a direct consequence of Theorem \ref{t:BRmain} and  Lemma~\ref{riccic0}.
To prove Theorem \ref{t:main2} we use the following Lemma to rewrite the condition in Theorem~\ref{t:main1}.
\begin{lemma}\label{t0u} We have the identity
\begin{equation}\label{bone-may} \mathrm{Ric}(X,X)- \sum_{\alpha=1}^3 R(X,I_\alpha X,I_\alpha X,X)=2nT^0(X,X)+(4n-8)U(X,X)+2(n-1)S.
\end{equation}
\end{lemma}
\begin{proof}
 Fix a unit vector $X$ and set $Y=Z=I_1X$, $V=X$ into \eqref{comp1}. One obtains the following  expression
\begin{align*}
3R(X,I_1X,I_1X,X)&+R(I_1X,X,I_1X,X)+R(I_2X,I_3X,I_1X,X)-R(I_3X,I_2X,I_1X,X)\\
&=2R(X,I_1X,I_1X,X)+2R(I_2X,I_3X,I_1X,X)\\
&=4(T^0(X,X)+T^0(I_1X,I_1X))
+8U(X,X)+4S.
\end{align*}
Do the same for $I_2$, $I_3$ and sum the obtained equalities, we obtain applying  the first Bianchi identity for the Biquard connection \eqref{bian01}.
\begin{align*}
2R(X,I_1X,I_1X,X)&+2R(I_2X,I_3X,I_1X,X)+2R(X,I_2X,I_2X,X)\\
&+2R(I_3X,I_1X,I_2X,X)+2R(X,I_3X,I_3X,X)+2R(I_1X,I_2X,I_3X,X)\\
&=2\sum_{\alpha=1}^3R(X,I_{\alpha}X,I_{\alpha}X,X)+2b(I_1X,I_2X,I_3X,X)\\
&=4\Big[3T^0(X,X)+\sum_{\alpha=1}^3T^0(I_{\alpha}X,I_{\alpha}X)\Big]+24U(X,X)+12S\\&=
8T^0(X,X)+24U(X,X)+12S.
\end{align*}
From \eqref{bian01} we have $b(I_1X,I_2X,I_3X,X)=2T^0(X,X)-6U(X,X)$.
Hence,
\begin{align}\label{seccurv}
\sum_{\alpha=1}^3R(X,I_{\alpha}X,I_{\alpha}X,X)&=6S+4T^0(X,X)+12U(X,X)-2T^0(X,X)+6U(X,X)\\
&=2T^0(X,X)+18U(X,X)+6S.\nn
\end{align}
 and the first identity of \eqref{sixtyfour} combined with \eqref{seccurv} gives \eqref{bone-may}.
%\begin{equation}\label{bone-may1} \mathrm{Ric}^{\nabla}(X,X)- \sum_{\alpha=1}^3 R(X,I_\alpha X,I_\alpha X,X)=2nT^0(X,X)+(4n-8)U(X,X)+2(n-1)S.
%\end{equation}
\end{proof}

Now, Theorem \ref{t:main2} follows from Theorem \ref{t:main1} and Lemma~\ref{t0u}.

\appendix

\section{Some technical facts and useful identities}
\label{s:quatapp}

Here we recall some properties of the torsion and curvature of the Biquard connection. See also \cite{Biq1,IMV,IV,IMV2,IPV1} for a comprehensive exposition. 
%\subsection{Torsion and curvature of quaternionic contact structures}
%
% 
%}

\subsection{Invariant decompositions}

Any endomorphism $\Psi$ of $\distr$ can be decomposed with respect to
the quaternionic structure $(\mathbb{Q},g)$ uniquely into four
$\mathrm{Sp}(n)$-invariant
parts %\begin{equation}\label{New4}
$\Psi=\Psi^{+++}+\Psi^{+--}+\Psi^{-+-}+\Psi^{--+},$ %\end{equation}
where $\Psi^{+++}$ commutes with all three $I_i$, $\Psi^{+--}$
commutes with
$I_1$ and anti-commutes with the others two, etc. The two $\mathrm{Sp}(n)\mathrm{Sp}(1)$%
-invariant components \index{$\mathrm{Sp}(n)\mathrm{Sp}(1)$-invariant
components!$\Psi_{[3]}$} \index{$\mathrm{Sp}(n)\mathrm{Sp}(1)$-invariant
components!$\Psi_{[-1]}$} are given by
\begin{equation*}
\Psi_{[3]}=\Psi^{+++},\quad
\Psi_{[-1]}=\Psi^{+--}+\Psi^{-+-}+\Psi^{--+}.
\end{equation*}
They are the projections on the eigenspaces of the Casimir
operator $%\begin{equation*}
\Upsilon =\ I_1\otimes I_1\ +\ I_2\otimes I_2\ +\ I_3\otimes I_3,
$
corresponding, respectively, to the eigenvalues
$3$ and $-1$, see \cite{CSal}. Note here that each of the three
2-forms $\omega_s$ belongs to the $[-1]$-component and constitute a
basis of the Lie algebra $\mathrm{sp}(1)$.

If $n=1$ then the space of symmetric endomorphisms commuting with
all $I_{\alpha}$ is 1-dimensional, i.e., the $[3]$-component of any
symmetric endomorphism $\Psi
$ on $\distr$ is proportional to the identity, $\Psi_{[3]}=-%
\frac{tr\Psi}{4}\id_{|H}$.

\subsection{The torsion tensor}

The torsion endomorphism $T_{\xi }=T(\xi ,\cdot ):H\rightarrow
H,\quad \xi \in V$ will be decomposed into its symmetric part
$T_{\xi }^{0}$ and skew-symmetric part $b_{\xi },T_{\xi }=T_{\xi
}^{0}+b_{\xi }$. Biquard
showed in \cite{Biq1} that the torsion $T_{\xi }$ is completely trace-free, $%
tr\,T_{\xi }=tr\,T_{\xi }\circ I_{{\alpha}}=0$, its symmetric part
has the properties 
\begin{gather*}
T_{\xi
_{{\alpha}}}^{0}I_{{\alpha}}=-I_{{\alpha}}T_{\xi
_{{\alpha}}}^{0},\quad  I_{2}(T_{\xi _{2}}^{0})^{+--}=I_{1}(T_{\xi
_{1}}^{0})^{-+-}, \\ I_{3}(T_{\xi _{3}}^{0})^{-+-}=I_{2}(T_{\xi
_{2}}^{0})^{--+}, \quad I_{1}(T_{\xi _{1}}^{0})^{--+}=I_{3}(T_{\xi
_{3}}^{0})^{+--}.
\end{gather*}
The skew-symmetric part can be represented as
$b_{\xi _{{\alpha}}}=I_{{\alpha}}U$, where $U
$ is a traceless symmetric (1,1)-tensor on $\distr$ which commutes with $%
I_{1},I_{2},I_{3}$. Therefore we have $T_{\xi _{{\alpha}}}=T_{\xi
_{{\alpha}}}^{0}+I_{{\alpha}}U$. When $n=1$ the tensor $U$
vanishes identically, $U=0$, and the torsion is a symmetric
tensor, $T_{\xi }=T_{\xi }^{0}$.
The two $\mathrm{Sp}(n)\mathrm{Sp}(1)$-invariant trace-free symmetric 2-tensors on
$\distr$
\begin{equation}  \label{Tcompnts}
T^0(X,Y)= g((T_{\xi_1}^{0}I_1+T_{\xi_2}^{0}I_2+T_{
\xi_3}^{0}I_3)X,Y) \ \text{ and }\ U(X,Y) =g(uX,Y)
\end{equation}
were introduced in \cite{IMV} and enjoy the properties
\begin{equation}  \label{propt}
\begin{aligned} T^0(X,Y)+T^0(I_1X,I_1Y)+T^0(I_2X,I_2Y)+T^0(I_3X,I_3Y)=0, \\
U(X,Y)=U(I_1X,I_1Y)=U(I_2X,I_2Y)=U(I_3X,I_3Y). \end{aligned}
\end{equation}
From \cite[Proposition~2.3]{IV} we have
\begin{equation}  \label{need}
4T^0(\xi_{\alpha},X,Y)=-T^0(I_{\alpha}X,Y)-T^0(X,I_{\alpha}Y),
\end{equation}
hence, taking into account \eqref{need} it follows
\begin{equation}  \label{need1}
T(\xi_{\alpha},X,Y)
=-\frac14\Big[T^0(I_{\alpha}X,Y)+T^0(X,I_{\alpha}Y)\Big]+U(I_{\alpha}X,Y).
\end{equation}

Any 3-Sasakian manifold has zero torsion endomorphism, and the
converse is true if in addition the qc-Einstein curvature (see
\eqref{qscs}) is a positive constant \cite{IMV}.

\subsection{Torsion and curvature}

Let $R=[\nabla,\nabla]-\nabla_{[\ ,\ ]}$ be the curvature tensor
of $\nabla$ and the dimension is $4n+3$. We denote the curvature
tensor of type (0,4)
and the torsion tensor of type (0,3) by the same letter, $%
R(A,B,C,D):=g(R(A,B)C,D),\quad T(A,B,C):=g(T(A,B),C)$, $A,B,C,D
\in \Gamma(TM)$. {The \emph{qc-Ricci tensor} $Ric$,
\emph{normalized qc-scalar curvature} $S$, \emph{qc-Ricci forms}
$\rho_{\alpha}$  of the Biquard connection are defined,
respectively, by the following formulas
\begin{equation}  \label{qscs}
\begin{aligned} &
\mathrm{Ric}(A,B)=\sum_{i=1}^{4n}R(X_i,A,B,X_i),\quad  4n\rho_{\alpha}(A,B)=\sum_{b=1}^{4n}R(A,B,X_i,I_{\alpha}X_i),\\ &
 8n(n+2)S=\sum_{i=1}^{4n}\mathrm{Ric}(X_i,X_i)=\sum_{i,j=1}^{4n}R(X_j,X_i,X_i,X_j),
%\zeta_s(A,B)=\frac1{4n}R(e_a,A,B,I_se_a).
\end{aligned}
\end{equation}
where $X_1,\dots,X_{4n}$ is an orthonormal basis of $\distr$.
}

A qc structure is said to be qc-Einstein if the horizontal
qc-Ricci tensor is a scalar multiple of the metric,
$%\begin{equation*}
\mathrm{Ric}(X,Y)=2(n+2)Sg(X,Y). $%\end{equation*}

As shown in \cite{IMV,IMV2} the qc-Einstein condition is
equivalent to the vanishing of the torsion endomorphism of the
Biquard connection. In this case $S$ is constant and the vertical
distribution is integrable. It is also worth recalling that the
horizontal qc-Ricci tensors and the integrability of the vertical
distribution can be expressed in terms of the torsion of the
Biquard connection \cite{IMV} (see also \cite{IMV1,IV,IPV1}). For
example, we have

\begin{equation}  \label{sixtyfour}
\begin{aligned} & \mathrm{Ric}(X,Y)
=(2n+2)T^0(X,Y)+(4n+10)U(X,Y)+2(n+2)Sg(X,Y),\\
& \rho_{\alpha}(X,I_{\alpha}Y) \  =\
-\frac12\Bigl[T^0(X,Y)+T^0(I_{\alpha}X,I_{\alpha}Y)\Bigr]-2U(X,Y)-%
Sg(X,Y),
\\
&
T(\xi_{\alpha},\xi_{\beta})=-S\xi_{\tau}-[\xi_{\alpha},\xi_{\beta}]_{|H},
\qquad S = -g(T(\xi_1,\xi_2),\xi_3),\\ &
g(T(\xi_{\alpha},\xi_{\beta}),X)=-\rho_{\tau}(I_{\alpha}X,\xi_{\alpha})=-\rho_{\tau}(I_{\beta}X,\xi_{\beta})=-g([\xi_{\alpha},%
\xi_{\beta}],X). \end{aligned}
\end{equation}
Note that for $n=1$ the above formulas hold with $U=0$.

\subsection{Bianchi identity}
We shall also need the first Bianchi identity and the general formula for the curvature
\cite{IMV,IV,IV2}.

The first Bianchi identity for the Biquard connection reads
\begin{align}  \label{bian01}
b(X,Y,Z,V)&=\sum_{(X,Y,Z)}R(X,Y,Z,V)=\sum_{(X,Y,Z)}\Bigl\{ (\nabla_XT)(Y,Z,V) +
T(T(X,Y),Z,V)\Bigr\}\\
&=\sum_{(X,Y,Z)}
T(T(X,Y),Z,V)=2\sum_{(X,Y,Z)}\sum_{\alpha=1}^3g(I_{\alpha}X,Y)
T(\xi_{\alpha},Z,V)\nn.
\end{align} where $\sum_{(X,Y,Z)}$ denotes the cyclic sum over $\{X,Y,Z\}$ and we used that  $(\nabla_XT)(Y,Z,V) =0$ for  horizontal vectors.

We also have the identities,  cf. \cite[Theorem 3.1]{IV} or \cite[Theorem 4.3.11]{IV2},
\begin{multline}\label{comp1}
 3R(X,Y,Z,V)-R(I_1X,I_1Y,Z,V)-R(I_2X,I_2Y,Z,V)-R(I_3X,I_3Y,Z,V)\\
= 2\Big[g(Y,Z)T^0(X,V)+g(X,V)T^0(Z,Y)-g(Z,X)T^0(Y,V)- g(V,Y)T^0(Z,X)\Big]
\\
-2\sum_{\alpha=1}^3\Big[\omega_\alpha(Y,Z)T^0(X,I_\alpha V)+\omega_\alpha(X,V)T^0(Z,I_\alpha Y)-\omega_\alpha(Z,X)T^0(Y,I_\alpha V)-
\omega_\alpha(V,Y)T^0(Z,I_\alpha X)\Big]
\\%\sum_{s=1}\Big[\omega_s(X,Y)\Big(8U(I_sZ,V)+4\rho_s(Z,V)\Big)-8\omega_s(Z,V)U(I_sX,Y)\Big]\notag\\
+\sum_{\alpha=1}^3\Big[2\omega_\alpha(X,Y)\Big(T^0(Z,I_\alpha V)-T^0(I_\alpha Z,V)\Big)-8\omega_\alpha(Z,V)U(I_\alpha X,Y)-4S
\omega_\alpha(X,Y)\omega_\alpha(Z,V)\Big];
\end{multline}
%Moreover
\begin{align}  \label{d3n5}
R(\xi _{\alpha},X,Y,Z)&=-(\nabla _{X}U)(I_{\alpha}Y,Z)+\omega
_{\beta}(X,Y)\rho
_{\tau}(I_{\alpha}Z,\xi _{\alpha}) \\\nn
&\ -\omega _{\tau}(X,Y)\rho _{\beta}(I_{\alpha}Z,\xi
_{\alpha})-\frac{1}{4}\Big[(\nabla
_{Y}T^{0})(I_{\alpha}Z,X)+(\nabla _{Y}T^{0})(Z,I_{\alpha}X)\Big] \\ \nn
&\ +\frac{1}{4}\Big[(\nabla _{Z}T^{0})(I_{\alpha}Y,X)+(\nabla _{Z}T^{0})(Y,I_{\alpha}X)%
\Big] -\omega _{\beta}(X,Z)\rho _{\tau}(I_{\alpha}Y,\xi _{\alpha}) \\ \nn
&\ +\omega _{\tau}(X,Z)\rho _{\beta}(I_{\alpha}Y,\xi
_{\alpha})-\omega _{\beta}(Y,Z)\rho _{\tau}(I_{\alpha}X,\xi
_{\alpha}) +\omega _{\tau}(Y,Z)\rho _{\beta}(I_{\alpha}X,\xi
_{\alpha}),
\end{align}%
where the Ricci two forms are given by
\begin{equation}
\begin{aligned} &6(2n+1)\rho_{\alpha}(\xi_{\alpha},X)=(2n+1)X(S)+\frac12
(\nabla_{e_a}T^0)[(e_a,X)-3(I_{\alpha}e_a,I_{\alpha}X)]\\&\hskip3.5in
-2(\nabla_{e_a}U)(e_a,X),\\ &
6(2n+1)\rho_{\alpha}(\xi_{\beta},I_{\tau}X)=(2n-1)(2n+1)X(S)-\frac{4n+1}{2}(
\nabla_{e_a}T^0)(e_a,X)\\ & \hskip1.5in
-\frac{3}{2}(\nabla_{e_a}T^0)(I_{\alpha}e_a,I_{\alpha}X)-4(n+1)(\nabla_{e_a}U)(e_a,X)
.\end{aligned}  \label{d3n6}
\end{equation}

\subsection{Fermi frame} \label{s:fermia}
Here we prove the existence of Fermi frame. We recall the statement

\begin{lemma}\label{fermi_framea}
Given a geodesic $\gamma(t)$, there exists a $\mathbb Q$-orthonormal frame, i.e., a horizontal
frame $X_i$, $i \in \{1,\dots,4n\}$,
%$$\{X_1,X_2=IX_1,\dots,X_{4n}=KX_{4n-3}\},$$ %
and vertical frame $\xi_{\alpha}$, $\alpha=1,2,3$
in a neighborhood of $\gamma(0)$, such that for all $\alpha, \beta
\in\{1,2,3\}$ and $i,j\in\{1,\dots,4n\}$,   
   \begin{itemize}
        \item[(i)] the frame is orthonormal for the Riemannian metric $g+\sum_{\beta}\eta_{\beta}^2$,
        \item[(ii)] $\nabla_{X_i} X_j|_{\gamma(t)} = \nabla_{\xi_\alpha} X_j|_{\gamma(t)} =\nabla_{X_i} \xi_\beta|_{\gamma(t)}=\nabla_{\xi_{\alpha}} \xi_\beta|_{\gamma(t)}
        =0$.
    \end{itemize}
     In particular, for all $\alpha, \beta,\tau 
\in\{1,2,3\}$ and $i,j\in\{1,\dots,4n\}$
$$((\nabla_{X_i}I_{\alpha})X_j)|_{\gamma(t)}=((\nabla_{X_i}I_{\alpha})\xi_{\beta})|_{\gamma(t)}=((\nabla_{\xi_{\beta}}I_{\alpha})X_j)|_{\gamma(t)}
=((\nabla_{\xi_{\beta}}I_{\alpha})\xi_{\tau})|_{\gamma(t)}=0.$$
\end{lemma}

\begin{proof}
Since $\nabla$ preserves the splitting $H\oplus V$ we can apply
the standard arguments for the existence of a Fermi normal frame
along a smooth curve with respect to a metric connection (see e.g.,
\cite{BIliev}). We sketch the proof for completeness.

Let $\{\tilde X_1,\dots,\tilde
X_{4n},\tilde\xi_1,\tilde\xi_2,\tilde\xi_3\}$ be a $\mathbb
Q$-orthonormal basis around $p=\gamma(0)$ such that $\tilde
X_{a_|p}=X_a(p)$ and $\tilde\xi_{{\alpha}_|p}=\xi_{\alpha}(p)$. We
want to find a modified  frame $X_a=o^b_a\tilde X_b$ and $\xi_{\alpha}=o^{\tau}_{\alpha}\tilde\xi_{\tau},$ which satisfies
the normality conditions along the smooth geodesic $\gamma(t)$.

Let $\varpi$ be the $\mathrm{sp}(n)\oplus \mathrm{sp}(1)$-valued connection 1-forms
with respect to the frame
$\tilde X_1,\dots,\tilde
X_{4n}$, and $\tilde\xi_1,\tilde\xi_2,\tilde\xi_3$, namely 
$$\nabla_{ D}\tilde
X_b=\varpi^c_b(D)\tilde X_c,\quad \nabla_{D}\tilde\xi_{\alpha}=\varpi^{\tau}_{\alpha}(D)\tilde\xi_{\tau},\qquad 
D\in\Gamma(TM).$$ Consequently, we have
$$\nabla_{ D} X_b=\omega^c_b(D)
X_c=[D(o^a_b)+o^d_b\varpi^a_d(D)](o^{-1})^c_aX_c,$$$$
\nabla_{D}\xi_{\beta}=\omega^{\tau}_{\beta}(D)\xi_{\tau}=[D(o^{\alpha}_{\beta})+o^{\nu}_{\beta}\varpi^{\alpha}_{\nu}(D)](o^{-1})^{\tau}_{\alpha}X_{\tau}$$

Since the Biquard connection preserves the splitting $H\oplus V$,
the existence of a Fermi normal frame along $\gamma(t)$ is
equivalent to the existence of a smooth solution to the system
$$[D(o^a_b)+o^d_b\varpi^a_d(D)]_{|_{\gamma(t)}}=0, \quad [D(o^{\beta}_{\alpha})+o^{\tau}_{\alpha}\varpi^{\beta}_{\tau}(D)]_{|_{\gamma(t)}}=0.$$
A smooth solution to this system on a small neighborhood along
$\gamma(t)$ exists, see e.g., \cite[Theorem 3.1]{BIliev}. Clearly,
the solution $o^a_b$ belongs to $\mathrm{Sp}(n), o^a_b\in \mathrm{Sp}(n)$ since the
connection 1-forms belong to the Lie algebra $\mathrm{sp}(n)$ and the
solution $o^{\beta}_{\alpha}\in \mathrm{Sp}(1)$ because the connection
1-forms are in the Lie algebra $\mathrm{sp}(1)$.
\end{proof}

\subsection*{Acknowledgements}
The first author has been supported by 
%the European Research Council, ERC StG 2009 ``GeCoMethods'', contract number 239748, by 
the  ANR project SRGI ``Sub-Riemannian Geometry and Interactions'', contract number ANR-15-CE40-0018. The second author has been partially supported by Contract DFNI
I02/4/12.12.2014 and Contract 195/2016 with the Sofia
University "St.Kl.Ohridski".

\bibliographystyle{abbrv}
\bibliography{quat}

\end{document}